\documentclass[11pt,a4paper,final]{amsart}
\usepackage[utf8]{inputenc}
\usepackage[T1]{fontenc}
\usepackage[english]{babel}
\usepackage[a4paper,margin=1in]{geometry}

\usepackage{amsmath, amsfonts, amssymb, amsthm, mathtools, indentfirst, bbm}
\usepackage[dvipsnames]{xcolor}
\usepackage[shortlabels]{enumitem}
\usepackage[languagenames,fixlanguage]{babelbib}
\usepackage{commath}
\usepackage{textcomp}
\usepackage{mathrsfs,dsfont}
\usepackage{diagbox}
\usepackage{booktabs}
\usepackage{array}
\allowdisplaybreaks

\newcommand{\normal}{\color{black}}

\usepackage{marginnote}

\makeatletter
\def\namedlabel#1#2{\begingroup
    #2%
    \def\@currentlabel{#2}%
    \label{#1}\endgroup
}
\makeatother
\usepackage[colorlinks=true,citecolor=red,urlcolor=blue,
linkcolor=red,bookmarksopen=true,unicode=true,pdffitwindow=false]{hyperref}

\usepackage{tikz}   % for drawings
\usepackage{tkz-euclide}

\theoremstyle{plain}
\newtheorem{theorem}{Theorem}[section]
\newtheorem{corollary}[theorem]{Corollary}
\newtheorem{lemma}[theorem]{Lemma}
\newtheorem{proposition}[theorem]{Proposition}

\theoremstyle{definition}
\newtheorem{remark}[theorem]{Remark}

\numberwithin{equation}{section}

\renewcommand\labelenumi{\textup{\alph{enumi})}}

\renewcommand\theenumi\labelenumi%{\textup{\alph{enumi})}}
\makeatletter\renewcommand{\p@enumii}{}\makeatother %\ref concatenates the \theenumi's with
\renewcommand{\leq}{\leqslant}

\renewcommand{\geq}{\geqslant}

\newcommand{\bbZ}{\mathbb{Z}}

\newcommand{\bbP}{\mathbb{P}}

\newcommand{\calU}{\mathcal{U}}
\newcommand{\calW}{\mathcal{W}}

\newcommand{\R}{\mathds{R}}
\newcommand{\N}{\mathds{N}}
\newcommand{\Pp}{\mathds{P}}

\newcommand{\Z}{\mathds{Z}}
\newcommand{\pr}{\mathds{P}}

\newcommand{\ex}{\mathds{E}}

\title[Harmonic functions for discrete Feynman--Kac operators]{Decay of harmonic functions for discrete time Feynman--Kac operators with confining potentials}
	
	\author{Wojciech Cygan}
		\address{
		Wojciech Cygan,
		Institut f\"{u}r Mathematische Stochastik,
		Technische Universit\"{a}t \newline Dresden, Germany 
		\& 
		Instytut Matematyczny,
		Uniwersytet Wrocławski, Poland
		}
		\email{wojciech.cygan@uwr.edu.pl}
		
	\author{Kamil Kaleta}
		\thanks{Research supported by National Science Centre, Poland, grant no.\ 2019/35/B/ST1/02421}
		\address{Kamil Kaleta, Faculty of Pure and Applied Mathematics,
						Wrocław University of Science and Technology, Poland}
		\email{kamil.kaleta@pwr.edu.pl}
		
		\author{Mateusz \'{S}liwi\'{n}ski}
		\address{Mateusz \'{S}liwi\'{n}ski, Faculty of Pure and Applied Mathematics,
						Wrocław University of Science and Technology, Poland}
		\email{mateusz.sliwinski@pwr.edu.pl}

	\date{}

\begin{document}

\begin{abstract}
We propose and study a certain discrete time counterpart of the classical  Feynman--Kac semigroup with a confining potential in countable infinite spaces. For a class of long range Markov chains which satisfy the direct step property	we prove sharp estimates for functions which are (sub-, super-)harmonic in infinite sets with respect to the discrete Feynman--Kac operators. These results are compared with respective estimates for the case of a nearest-neighbour random walk which evolves on a graph of finite geometry. We also discuss applications to the decay rates of solutions to equations involving graph Laplacians and to eigenfunctions of the discrete Feynman--Kac operators. We include such examples as non-local discrete Schr\"odinger operators based on fractional powers of the nearest-neighbour Laplacians and related quasi-relativistic operators.  Finally, we analyse various classes of Markov chains which enjoy the direct step property and illustrate the obtained results by examples.
\end{abstract}



\maketitle
%\tableofcontents

\section{Introduction} 

The main goal of this paper is to develop the theory of discrete time Feynman--Kac semigroups with general confining potentials which we define for Markov chains with values in infinite discrete spaces. We focus on chains which exhibit a certain long range distributional   property -- the  \emph{direct step property}  (DSP in short). It   means that the two-step transition probability   of the chain   is dominated (up to a multiplicative constant) by the one-step transition  probability.
 The first part of the paper   is concerned with   the decay properties of functions that are harmonic (resp.\ subharmonic, superharmonic) in an infinite subset of the space with respect to the Feynman--Kac operator related to a Markov chain satisfying the DSP.
We then compare our results with the case of nearest-neighbor random walks evolving in graphs of finite geometry. In the second part  we focus on the DSP property itself. We discuss various methods which allow us to construct Markov chains with the DSP,   putting special emphasis on   the technique of the \emph{discrete subordination}. We illustrate this by numerous examples, showing that in fact the DSP class includes many important chains that were already studied in the literature. Results obtained here are fundamental for our ongoing project where we analyse further analytic properties of the discrete time Feynman--Kac semigroups. 

\subsection*{Motivation}
Our motivations for this project are two-fold. The first   one originates   from the theory of non-local Schr\"odinger operators in $L^2(\R^d)$, while the second comes from the theory of discrete-time Markov chains evolving in countable infinite spaces. We now briefly describe each of these two paths and we   point out   some connections between them.

Let $L$ be an $L^2$-generator of a symmetric L\'evy process $(X_t)_{t \geq 0}$ in $\R^d$ (\emph{L\'evy operator}) \cite{ Bottcher_Schilling_Wang, Jacob} and let $V:\R^d \to \R$ be a locally bounded function such that $V(x) \to \infty$ as $|x| \to \infty$ (\emph{confining potential}). 
The \emph{Schr\"odinger operator} is then defined as
$$
H = -L+V, \quad \text{acting in} \ L^2(\R^d).
$$ 
Prominent examples include the following operators $L$ (and the corresponding stochastic processes):

\noindent 
a) \emph{classical Laplacian}: $L=\Delta$ (Brownian motion running at twice the speed); 

\noindent 
b) \emph{fractional Laplacian}: $L=-(-\Delta)^{\alpha/2}$, $\alpha \in (0,2)$ (isotropic $\alpha$-stable L\'evy process); 

\noindent
c) \emph{quasi-relativistic operator}: $L=-(-\Delta+m^{2/\alpha})^{\alpha/2}+m$, $\alpha \in (0,2)$, $m>0$ (isotropic relativistic $\alpha$-stable L\'evy process). 

\noindent
We recall that these and many other examples of $L$'s can be constructed via the spectral theory, that is the operator $L$ can be written as $L = -\phi(-\Delta)$, where $\phi$ is a Bernstein function such that $\phi(+0)=0$, cf. \cite{Schilling_book}. 
Note that with this approach we can obtain local as well as non-local operators. For instance, $L$ in a) is local as it is a second order differential operator, while both $L$'s in b) and c) are non-local -- this is because of the jumping nature of the corresponding L\'evy processes.
It is remarkable that such operators   and the related processes have numerous applications in physical sciences \cite{Woyczynski}; in view of the confinement assumption the corresponding Schr\"odinger operators $H$ usually serve as Hamiltonians in various mathematical models of \emph{oscillators} in non-relativistic and (quasi-)relativistic quantum mechanics (see, e.g. \cite{Durugo_Lorinczi, Garbaczewski_Stephanovich, Gatland, osc2, Mohazzabi} and references in these papers). 

The Schr\"odinger semigroup $\big\{e^{-tH}: t \geq 0\big\}$ admits the following stochastic representation
which is given in terms of the L\'{e}vy process generated by $L$
\begin{equation}\label{eq:F-K-formula}
e^{-tH} f(x) = \ex^x\left[e^{-\int_0^t V(X_s)ds} f(X_t)\right], \quad f \in L^2(\R^d),\quad t>0.
\end{equation}
This equality is called the \emph{Feynman--Kac formula} \cite{Demuth_Casteren}. It is known to be a powerful tool which allows one to study various  analytic properties of these operators by means of probabilistic methods.
Recent contributions include estimates of the heat kernel, heat content and trace \cite{Acuna, Acuna_Banuelos, Banuelos_Yolcu, Jakubowski_Wang, Kaleta_Schilling, Wang}, harmonic functions, ground states, eigenfunctions and eigenvalues, and spectral bounds \cite{Jacob_Wang, Kaleta, Kaleta_Lorinczi_3, Kulczycki, Takeda}, intrinsic hyper- and ultracontractivity \cite{Chen_Wang, Kaleta_Lorinczi_Kwasnicki, Kulczycki_Siudeja}, to mention just a few of them.  

In this paper we study semigroups which are given by a pure discrete counterpart of the right-hand side of \eqref{eq:F-K-formula}, i.e.\ when the underlying processes are discrete-time Markov chains taking values in a countable infinite state space $X$. Generators of such semigroups are certain normalizations of discrete Schr\"odinger operators  (they act on function spaces over $X$) and they are defined through the generators of the Markov chains -- this can be understood as a discrete-time Feynman--Kac formula. This correspondence provides direct access to various properties of objects related to non-local discrete Schr\"odinger operators  which are exploited via elementary methods based on discrete time evolution semigroups and processes.   In this paper we apply this approach to study the decay properties of harmonic functions, but in fact it has more far-reaching consequences.   Our investigations concentrate on the class of Markov chains with the DSP.   This framework covers many interesting examples of discrete counterparts of non-local Schr\"odinger operators that were studied in the Euclidean case. 

One can also look at our investigations from a different perspective.  
The discrete-time Feynman--Kac semigroups with confining potentials serve as transition semigroups of discrete-time Markov chains evolving in countable infinite spaces, whose paths are killed with random intensity given by the potential. This    killing effect intensifies at infinity, leading to a variety of interesting long-range and limiting phenomena, especially for the underlying discrete-time processes that satisfy the DSP.
One of the main goals of this project is to understand the long-time asymptotic and ergodic properties of such Feynman--Kac semigroups and the corresponding processes evolving in the presence of the killing Schr\"odinger potentials. In this context, we want to mention here a recent work by Diaconis, Houston-Edwards and Saloff-Coste \cite{Diaconis} which gave us some new insight and motivation.  
The two aforementioned motivations are strongly connected to each other -- this is manifested via the probabilistic background lying behind  the analytic approach which we undertake in this article.

Below we briefly display the setting and our main results, together with the references to the corresponding theorems in the remaining part of the text.

\subsection*{Discrete time Feynman--Kac semigroups}
Let $X$ be a countable infinite set and let $P : X \times X \to [0,1]$ be a (sub-)probability kernel, that is
\begin{align} \label{def:prob_kernel}
    \sum_{y \in X} P(x,y) \leq 1, \ \ \text{for every} \ x \in X. 
\end{align}
Equivalently, there is a time-homogeneous Markov chain $\big\{Y_n: n \in \N_0\big\}$, defined on a given probability space $(\Omega, \mathcal F, \pr)$, with values in $X$ and one-step transition probabilities given by
$$
\pr(Y_{n+1}=y \, \mid \, Y_n=x) = P(x,y).
$$
Throughout we use the standard notation for the measure of the process starting at $x \in X$, that is  $\pr^x(Y_n =y) := \pr(Y_n=y \, \mid \, Y_0=x).$ 
The corresponding expected value is denoted by $\ex^x$. When the sum in \eqref{def:prob_kernel} is equal to $1$ for every $x \in X$, the process $\big\{Y_n: n \in \N_0\big\}$ is conservative in the sense that it has a full probability measure $\pr^x$ for every $x \in X$. Otherwise, it can be interpreted as a \emph{killed} process. We can complete its law to a full probability measure by the standard procedure which is based on adding an extra \emph{cemetary} point $\partial$ to the state space $X$ and extending $P$ to $X_{\partial} \times X_{\partial}$, where $X_{\partial} = X \cup {\partial}$. In this paper, however, we do not follow this  path -- we  allow the kernel $P$ to be strictly sub-probabilistic. Let us remark that we do  not assume symmetry of $P$.

Let $V:X \to \R$ be a function such that $\inf_{x \in X} V(x) > 0$ and let us introduce a semigroup of operators $\{\calU_n: n \in \N_0\}$ defined as 
\begin{align} \label{eq:FK-discrete}
\calU_0 f = f ,\quad \calU_nf(x) = \ex^x \bigg[\prod_{k=0}^{n-1}\frac{1}{V(Y_k)}  f(Y_n) \bigg], \quad n\geq 1,
\end{align}
for any admissible function $f$. Observe that
$
\calU_n f = \calU^{n} f, n\geq 1,
$
where $\calU^{n}$ denotes the $n$\textsuperscript{th} power of the operator
\begin{align}\label{def:calU}
\calU f(x) = \frac{1}{V(x)} \sum_{y \in X} P(x,y)f(y), \quad x \in X.
\end{align}
The formula \eqref{eq:FK-discrete} can be seen as a discrete time and space counterpart of \eqref{eq:F-K-formula}. We therefore call $\{\calU_n: n \in \N_0\}$ the \emph{discrete time Feynman--Kac semigroup} with potential $V$ associated with the chain $\{Y_n: n \in \N_0\}$. Observe that the discrete time multiplicative functional under the expectation in \eqref{eq:FK-discrete} could also be alternatively defined as $\prod_{k=1}^{n}\frac{1}{V(Y_k)} $ which would lead to a different semigroup. This results in the duality structure and this issue is discussed in more detail in Section \ref{sec:adjoint_F-K}.

The study of multiplicative functionals such as appearing in \eqref{eq:FK-discrete}, for processes with discrete time parameter, has a long history. This is mainly related to the famous observation by Mark Kac that various Wiener functionals can be effectively approximated by their certain discretizations \cite{Kac_TAMS_49, Kac_Berkley_Symposium}. Such techniques turned out to be powerful tools in the study of boundary value problems for classical Schr\"odinger operators on bounded domains of $\R^d$ for which the solutions are given by the classical Feynman--Kac formula. 
Similar questions have been raised for the simple random walk evolving in $\Z^d$, equipped with its natural Cayley graph structure, by Cs\`aki \cite{Csaki} for the one-dimensional case, and by Anastassiou and Bendikov in \cite{Anastassiou_Bendikov} for the multidimensional (parabolic) case. 

The operator $\calU -I$ is the central object in the present paper. We call it the \emph{Feynman--Kac operator}.
It can be directly checked that if $f$ is a function on $X$ such that $\calU_n |f| (x) < \infty$, for any $x \in X$ and $n \in \N_0$ (e.g.\ if $f$ is bounded), then
$
u(n,x) = \calU_n f(x)
$
is the unique solution to the following Cauchy problem
$$
\begin{cases}
\partial_n u(n,x) = (\calU -I)_x u(n,x) \\
u(0,x) =  f(x), 
\end{cases}
$$
where $\partial_n u(n,x) = u(n+1,x) - u(n,x)$ is the first-difference operator.

An important link to the classical theory is the observation that the operators $\calU -I$ can be seen as certain normalizations of the discrete Schr\"odinger operators. More precisely, if
$$
H f(x) = \sum_{y \in X} P(x,y)\big(f(x)-f(y)\big) + V(x)f(x),
$$
where $V:X \to \R$ is a \emph{potential} such that $\inf_{x \in X} ( V(x) +\sum_{y \in X} P(x,y)) > 0$, then
$$
\frac{1}{V(x) + \sum_{y \in X} P(x,y)} H f(x) = (I- \calU) f(x),
$$
where the operator $\calU$ is defined with the shifted potential $V(x)+\sum_{y \in X} P(x,y)$. It is therefore evident that the operators $H$ and $I-\calU$ share many analytic properties. In the present paper we exploit the fact that they have joint harmonic functions,  see Section \ref{sec:graph_L}. This idea has been used very recently by Fischer and Keller in the study of the Riesz decomposition for superharmonic functions of graph Laplacians \cite{Fischer_Keller}. 

\subsection*{Results for Markov chains with the DSP and confining potentials}

We obtain results for a class of Markov chains with a certain long range distributional property. Recall that the probability to move from $x$ to $y$ in $n$ steps is inductively defined as 
$$
P_n(x,y) = \sum_{z \in X} P(x,z)P_{n-1}(z,y), \quad n>1.
$$
We assume that the kernel $P(x,y)$ satisfies the following regularity condition:
	\begin{itemize}
\item[\bfseries(\namedlabel{A}{A})] 
We have $P(x,y) >0$, for all $x,y \in X$, and there exists a constant $C_*>0$ such that
\begin{align}\label{eq:DSP}
 P_2(x,y) \leq C_* P(x,y), \quad x, y \in X.
\end{align} 
\end{itemize}

\noindent
Condition \eqref{eq:DSP} has an interesting heuristic interpretation: \emph{the probability to move from $x$ to $y$ in two consecutive steps is asymptotically smaller than the probability to move in the one direct step}. For this reason we call this condition the \emph{direct step property} (DSP in short). It should be emphasized that the rate of domination in the DSP does not depend on $x$ and $y$ (recall that $X$ is infinite). Clearly, this property extends to the $n$-step transition probability, that is $P_n(x,y) \leq C_*^{n-1} P(x,y)$, $x,y \in X$. Observe that under \eqref{eq:DSP} the positivity of the kernel $P$ in assumption \eqref{A} is in fact equivalent to a weaker condition that the Markov chain associated with $P$ is irreducible, i.e.\ for every $x,y \in X$ there exists $n \in \N$ such that $P_n(x,y)>0$. We remark in passing that the DSP can be seen as a discrete counterpart of the direct jump property (DJP) -- the condition on a L\'evy measure which is a useful tool in the study of jump L\'evy processes in $\R^d$ (see \cite{Kaleta_Schilling} and references therein). The condition of this type has been first proposed by Kl\"uppelberg \cite{Kluppelberg} for distributions on the half-line.

In the paper we consider the class of \textit{confining} potentials $V$, that is  satisfying the following condition

	\begin{itemize}
\item[\bfseries(\namedlabel{B}{B})] 
For every $M>0$ there exists a finite set $B_M \subset X$ such that $V(x) \geq M$ for $x \in B_M^c$.
\end{itemize} 

\medskip
\noindent
An admissible function $f$ is called $(\calU-I)$-harmonic in a set $D\subset X$ ($(\calU-I)$-subharmonic, $(\calU-I)$-superharmonic, respectively) if $(\calU-I)f(x)=0$ for $x\in D$ ($\geq 0$, $\leq 0$, resp.).
The estimates for harmonic functions which we prove in the DSP case in Section \ref{sec:est_dsprw} can be summarized as follows.

\medskip
\noindent
(1) \textsl{Upper bound for subharmonic functions}: Under the DSP, assumption \eqref{B} forces that there exists a finite set $B_0 \subset X$ such that for any finite set $B \subset X$ with $B \supseteq B_0$ and for any non-negative and bounded function $f$ which is $(\calU -I)$-subharmonic in $B^c$ we have
$$ 
f(x) \leq C \, \frac{P(x,x_0)}{V(x)}  \sum_{y \in B} f(y), \quad x \in B^c, \ x_0 \in B, 
$$
with a constant $C=C(P,B)$ which is independent of $V$ and $f$, see Theorem \ref{thm:main_1}. The proof of this result is transparent and quite elementary. It is based on a tricky self-improving estimate which combines the DSP with assumption \eqref{B}.
We remark that in many cases the set $B_0$ and the constant $C$ can be given explicitly. 

The matching lower bound for superharmonic functions is obtained in a slightly more general setting and it also indicates that the upper bound in (1) is sharp.

\medskip
\noindent
(2) \textsl{Lower bound for superharmonic functions}: Under assumption \eqref{A}, for any set $D \subset X$ and any nonnegative function $f$ which is $(\calU -I)$-superharmonic in $D$, we have for any finite set $B\subset X$, 
$$
f(x) \geq \widetilde C \frac{P(x,x_0)}{V(x)} \sum_{y \in B} f(y), \quad x \in D \cap B^c,\ x_0 \in B,
$$
with a constant $\widetilde C= \widetilde C(P,B)$ which is independent of $V$, $f$ and $D$, see Proposition \ref{prop:lower_bound}. 
In this case the finite set $B$ is arbitrary. Similarly as for the upper bound, in many cases the value of the constant $\widetilde C$ can be given explicitly. 

A combination of our results from Theorem \ref{thm:main_1} and Proposition \ref{prop:lower_bound} (presented in (1)--(2) above) gives the two-sided sharp estimates for harmonic functions.

\medskip
\noindent
(3) \textsl{Two-sided estimate for harmonic functions}: Under assumptions \eqref{A} and \eqref{B}, there exists a finite set $B_0 \subset X$ such that for any finite set $B \subset X$ with $B \supseteq B_0$, any set $D \subset X$ and any non-negative non-zero and bounded function $f$ which is $(\calU -I)$-harmonic in $ D \cap B^c$ and such that $f(x) = 0$ for $x \in D^c \cap B^c$ we have
\begin{align}\label{eq:uBHI}
\widetilde C \leq \frac{f(x)}{ \frac{P(x,x_0)}{V(x)} \sum_{y \in B} f(y)} \leq C, \quad x \in D \cap B^c, \  x_0 \in B.
\end{align}
This can be seen as a discrete counterpart of the estimates proved in \cite[Theorem 2.2]{Kaleta_Lorinczi_1} in the case of L\'evy processes. 
As a direct consequence of \eqref{eq:uBHI} we obtain  \normal the following result.

\medskip
\noindent
(4) \textit{Uniform Boundary Harnack Inequality at infinity}: if $f$ and $g$ are two nonzero $(\calU -I)$-harmonic functions as in (3), then 
\begin{align}\label{ineq:BHI}
\left(\frac{\widetilde C}{C}\right)^2 \leq \frac{f(x)g(y)}{g(x)f(y)} \leq  \left(\frac{C}{\widetilde C}\right)^2 ,  \quad x, y \in D \cap B^c.
\end{align}
Result from (3) and (4) are given in Corollary \ref{cor:BHI_infty}.  
The inequality in \eqref{ineq:BHI} is a discrete version of the uniform Boundary Harnack Inequality (uBHI) at infinity which is a fundamental theorem in the potential theory of continuous time Markov processes and their generators. The word ``uniform'' refers to the fact that the constants appearing in the estimates do not depend on $D$ and $V$, and the finite set $B_0$ depends on $V$ only through its rate of growth at infinity (this means that if $B_0$ is appropriate for a given $V$, then it is also appropriate for any $\widetilde V$ such that $\widetilde V \geq V$). BHI has been widely studied for both local and non-local operators and the corresponding processes on bounded domains of $\R^d$. We refer the reader to the paper by Bogdan, Kumagai and Kwa\'{s}nicki \cite{Bogdan_Kumagai_Kwasnicki} for  general results on jump Feller processes, an excellent overview of the history, references and discussion on applications of BHI. Our present estimate \eqref{eq:uBHI} can be understood as a discrete time and space variant of the inequality stated by Kwa\'{s}nicki for jump isotropic $\alpha$-stable processes in $\R^d$ and $V \equiv 0$ \cite[Corollary 3]{Kwasnicki}. It was derived from the general result proven by Bogdan, Kulczycki and Kwa\'{s}nicki in \cite{Bogdan_Kulczycki_Kwasnicki}. Recently, Kim, Song and Vondra\v{c}ek \cite{Kim_Song_Vondracek} obtained a version of BHI at infinity for jump Feller processes on metric measure spaces.

We remark that all of our results presented in (1)--(4) can be extended beyond the setting of (sub-)probability kernels (for details see Remark \ref{rem:ext}).

\subsection*{Related estimates for nearest-neighbor walks with confining potentials}

It is instructive to compare our results obtained for Markov chains with the DSP with corresponding estimates for nearest-neighbor random walks evolving in connected graphs of finite geometry. The necessary set-up is precisely described at the beginning of Section \ref{sec:est_nnrw}.

Asymptotic properties of long-range random walks usually differ substantially from those of nearest-neighbor walks. 
This is also the case in the present situation -- under the killing effect (coming from the confining potential) on the paths of the underlying stochastic process the discrepancy between the decay rates is particularly evident. 
As we could not find the result of this type in the literature, we provide the respective estimates under the assumption that the potential is isotropic and increasing with respect to the \emph{graph} (\emph{geodesic}) \emph{metric} $d$ in $X$ which we equip with the graph structure. Our results can be summarized as follows. 
 
\medskip
\noindent
(1) \textsl{Upper bound for subharmonic functions.}  We obtain an upper estimate for bounded non-negative functions that are $(\calU -I)$-subharmonic in the complement of a geodesic ball. The decay rate of such a function $f(x)$ is governed by the expression of the form $\prod_{i=0}^{d(x,x_0)} (1/V(x_i)$ which is evaluated along the shortest path $x_0 \rightarrow x_1 \rightarrow \dots \rightarrow x$ connecting a given fixed point $x_0$ with $x$ in the graph over $X$, see Theorem \ref{th:upper_subharm_nnrw}.

\medskip
\noindent
(2) \textsl{Lower estimate for superharmonic functions and related two-sided bound.} In Theorem \ref{th:lower_est_nnrw} we obtain a lower bound for non-negative functions which are $(\calU -I)$-superharmonic in an unbounded, connected and geodesic convex subset of $X$. This estimate differs from that  in (1) by an extra multiplicative constant under the product. 
\medskip

In the case when the potential $V$ grows regularly enough at infinity  then the decay of a bounded and nonnegative $(\calU -I)$-harmonic function is governed  by the expression  of the form
\begin{align} \label{eq:two-sided-nnrw} 
e^{-c \, d(x,x_0) \log V(x) (1+o(1))}, \quad \text{as} \ \ d(x,x_0) \to \infty.
\end{align}
For a precise statement see Corollary \ref{cor:nnrw_strongest}. 
Observe that the results obtained for Markov chains with the DSP are in sharp  contrast to the decay rate obtained in \eqref{eq:two-sided-nnrw}. We refer the reader to Section \ref{sec:expl_rates} for some explicit examples. 

\subsection*{Direct applications}
We now discuss two specific applications of the presented estimates. 

\medskip
\noindent
(1) \textsl{Applications to equations involving the graph Laplacians.} 
Our results can be effectively applied to study the decay properties of solutions to the equation $H f(x) = 0$, $x \in D$, where $H$ is the graph Laplacian in the graph over $X$ and the set $D \subset X$ is infinite. More precisely, 
if $\{b(x,y)\}_{x,y\in X}$ is a family of weights over edges in $X$, as explained in Section \ref{sec:graph_L}, and  if $m$ is a positive measure on $X$
and $V : X \to \R$ is a potential satisfying assumption \eqref{B} then the operator defined as
$$
H f (x) = \frac{1}{m(x)} \sum_{y \in X} b(x,y) \big(f(x)-f(y)\big) + V(x) f(x),
$$
for $f$ such that $  \sum_{y} b(x,y) |f(y)| < \infty$ for every $x \in X$, is called the \emph{graph Laplacian}. Such operators can be seen as discrete anologues of Schr\"odinger operators with confining potentials. For an excellent account of the theory and an overview of recent contributions in the area of operators on infinite graphs we refer the reader to the monograph by Keller, Lenz and Wojciechowski \cite{keller_lenz_wojciechowski_2021}. 

The action of $H$ on functions that are supported in the complement of a bounded set $A$ can be reduced to a certain normalization of the operator $I - \calU$, which is constructed via the sub-probability kernel
$P(x,y)= b(x,y)/\sup_{x \in X} \sum_{y} b(x,y)$ and the potential $\widetilde V$ determined by the initial data $V$, $b$ and $m$. In consequence, these two operators share functions which are harmonic in subsets of $A$ (see Proposition \ref{prop:H_and_A} and the discussion following it). Therefore, under some mild assumptions on $b$ and $m$, our results can be applied to obtain estimates for functions harmonic with respect to $H$ in infinite subsets of $X$ in two cases: for weights $b(x,y)$ which lead to the DSP probability kernels $P(x,y)$ (Corollary \ref{thm:main_H}) and for $b(x,y)$ of the nearest-neighbor type (Corollary \ref{thm:main_H_nnrw}). Moreover, in the first case the BHI at infinity holds. This seems to be of special interests in the $\ell^2$-setting, since the operator $H$ has a specific meaning in quantum physics.  For detailed statements and further discussion we refer to Section \ref{sec:graph_L}.

\medskip
\noindent
(2) \textsl{Estimates for eigenfunctions of discrete time Feynman--Kac semigroups}. Suppose there is a positive measure $\mu$ on $X$. Under some natural assumptions on the kernel $P(x,y)$, the discrete time Feynman--Kac semigroup $\{\calU_n: n \in \N_0\}$ consists of operators that are bounded in $\ell^2(X,\mu)$ and act as bounded operators from $\ell^2(X,\mu)$ to $\ell^{\infty}(X,\mu)$ (it means they are ultracontractive). Moreover, under condition \eqref{B}, the operators $\calU_n$ are compact in $\ell^2(X,\mu)$ (Lemma \ref{lem:cpt}). In particular, they have purely discrete spectra and the ground state of the operator $I-\calU$ exists. As we already know that (due to ultracontractivity) any $\ell^2$-eigenfunction is bounded, our pointwise estimates from Section \ref{sec:gen_est} apply directly. This is discussed in more detail in Section \ref{sec:est_ef_FK}. 

As a  valuable gain we obtain sharp two-sided estimates for the ground state eigenfunction outside of a finite subset of $X$. This result has many far-reaching consequences. In particular, it is fundamental for further developments in the theory of discrete time Feynman--Kac semigroups with confining potentials. In our ongoing work we apply it to find sharp two-sided estimates for the kernel of the operator $\calU_n$ and to characterize the intrinsic contractivity properties. 
These results can in turn be used to analyse further long-time properties of the corresponding semigroups. 
It is rather a general rule that a sufficiently  detailed knowledge of the ground state enables us to give a precise description of the large-time behaviour of the corresponding semigroup, see e.g.\ Kaleta and Schilling \cite{Kaleta_Schilling} for a recent development in the theory of Schr\"odinger semigroups for L\'evy operators and Diaconis, Houston-Edwards and Saloff-Coste \cite[Sections 7.1-7.2]{Diaconis} for recent results in the case of discrete-time Markov chains with finite state spaces (see also Remark 7.19 and Examples 7.20--7.22 in \cite{Diaconis})

\subsection*{Markov chains with the DSP and discrete subordination}
We are finally concerned with the question: is the class of Markov chains with the DSP rich enough? This is partially answered in Section \ref{sec:DSP_chains}. We show that if we equip the space $X$ with a metric $d$ and if the sub-probability kernel $P$ depends on the distance and it is comparable with a doubling function $J$, that is $P(x,y)\asymp J(d(x,y)) $, then such kernel $P$ satisfies the DSP, see Proposition \ref{prop:doubling}. This condition includes many important examples of long-range random walks, for instance stable-like random walks in the integer lattice, see e.g.\ Bass and Levin \cite{Bass_Levin}, as well as random walks in measure metric spaces studied recently by Murugan and Saloff-Coste in \cite{Murugan_Saloff_2} and \cite{Murugan_Saloff-Coste}. In Corollary \ref{cor:JandK} we also extend this observation to kernels with much lighter tails.

On the other hand, we  establish a useful result which states that the DSP is stable under random change of time. To be more precise, if we consider an increasing random walk $\{\tau_n : n\in \N_0\}$ with values in $\N_0$ which satisfies the DSP and if $\{Z_n : n\in \N_0\}$ is an independent of $\tau_n$ homogeneous Markov chain in $X$ then the time-changed Markov chain $\{ Z_{\tau_n} : n\in\N_0\}$ enjoys the DSP as well, see Lemma \ref{lem:subord}. This is a powerful method which allows one to construct a number of examples of Markov chains satisfying our assumption \eqref{A}, see Corollary \ref{cor:sufficient} (in Lemma \ref{lem:suff_DSP_N} we also give an easy-to-check sufficient condition for the walk $\{\tau_n : n\in \N_0\}$ to satisfy the DSP). We exploit this construction with the aid of the discrete subordination which was developed by Bendikov and Saloff-Coste in \cite{Bendikov-Saloff} for random walks on groups.  As admissible random time-change processes they admitted a specific class of random walks $\tau_n$ whose one-step distributions are uniquely determined (through their Laplace transforms) by a Bernstein function $\phi $ such that $\phi (0+)=0$ and $\phi (1)=1$. If $\{Z_n : n\in \N_0\}$ is the standard nearest neighbour walk in $X$ equipped with a graph structure then the generator of the subordinate Markov chain $\{ Z_{\tau_n} : n\in\N_0\}$ is of the form $-\phi(-\Delta)$, where $\Delta$ stands for the classical discrete Laplacian. This enables us to study various important non-local discrete counterparts of operators known from the $L^2(\R^d)$-theory, such as \emph{fractional} Laplacians and \emph{quasi-relativistic} operators.

In particular, we investigate the class of Markov chains associated with Bernstein functions $\phi (\lambda) = \lambda^\alpha$, for $\alpha \in (0,1)$ -- it results in an $\alpha$-\textit{stable} subordinator, and $\phi (\lambda) = (\lambda +m^{1/\alpha})^\alpha -m$, for $\alpha \in (0,1)$ and $m\geq 0$ -- this gives a \textit{relativistic} $\alpha$-stable subordinator (see Propositions \ref{prop:rel_stab} -- \ref{prop:UHK_relativistic}). These specific examples of long range Markov chains 
may be of special interest in mathematical physics and modelling. 

Finally, we discuss a handy method of constructing Markov chains with the DSP on product spaces, including integer lattices and products of more general graphs, see Section \ref{sec:product}.  

To illustrate our estimates of harmonic functions for discrete Feynman--Kac operators we collect in Section \ref{sec:expl_rates} some explicit examples of the decay rates which are derived for various Markov chains and confining potentials

\section{Estimates for functions harmonic in infinite sets} \label{sec:gen_est}

In this section we present estimates for functions which are subharmonic and superharmonic with respect to the discrete Feynman-Kac operators. We also study the decay of functions that are harmonic outside of a finite set. 

Recall that a function $f$ is called $(\calU -I)$-harmonic ($(\calU -I)$-superharmonic, $(\calU -I)$-subharmonic, resp.) in a non-empty set $D\subset X$ if $(\calU -I) f (x) = 0$ ($\leq 0$, $\geq 0$, resp.) for $x\in D$.

\subsection{Estimates for Markov chains with the DSP}\label{sec:est_dsprw}
In this section we find estimates of harmonic functions for the class of Markov chains satisfying our assumption \eqref{A}: $P(x,y)> 0$ for all $x,y\in X$ and there is a constant $C_{\ast}>0$ such that
\begin{align*}%\label{DJP}
 P_2(x,y)\leq C_{\ast}P(x,y),\quad x,y\in X.
\end{align*}

We first show that the kernel $P(x,y)$ can be uniformly localized in the second variable. For every finite set $B \subset X$ we define
\begin{align*}
\underline{K}_B:= \inf \left\{\frac{P(x,y)}{P(x,z)}: x \in X; y, z \in B\right\}
\quad \mathrm{and} \quad 
\overline{K}_B := \sup \left\{\frac{P(x,y)}{P(x,z)}: x \in X; y, z \in B\right\}.
\end{align*}

\begin{lemma} \label{lem:comparable}
Under assumption \eqref{A}, for every finite set $B \subset X$ we have $0 < \underline{K}_B \leq \overline{K}_B < \infty$.
\end{lemma}

\begin{proof}
It follows from \eqref{A} that for every $x \in X$ and $y, z \in B$ we have
$$
0 < P(x,z)P(z,y) \leq \sum_{w \in X} P(x,w)P(w,y) \leq C_{\ast} P(x,y),
$$
$$
0 < P(x,y)P(y,z) \leq \sum_{w \in X} P(x,w)P(w,z) \leq C_{\ast} P(x,z).
$$
This immediately implies
$$
\underline{K}_B \geq \frac{\inf_{y, z \in B}P(z,y)}{C_{\ast}} > 0 \quad \text{and} \quad \overline{K}_B \leq \frac{C_{\ast}}{\inf_{y, z \in B}P(y,z)}< \infty
$$
which completes the proof.
\end{proof}

In the remaining part of this section, we fix a finite set $B_0 \subset X$ such that
\begin{align} \label{choice_r_0}
C_1:=\sup \left\{\frac{1}{V(x)}: x \in B_0^c \right\}  < 1 \wedge \frac{1}{C_{\ast}}.
\end{align} 
The existence of such a set is secured by assumption \eqref{B}. Note that $B_0$ depends on $V$ and $P$.

Our first main result is the following upper bound for functions that are $(\calU -I)$-subharmonic in infinite sets.

\begin{theorem}\label{thm:main_1} 
Under assumptions \eqref{A} and \eqref{B}, there exists a constant $C_2>0$ such that for any finite set $B \subset X$ with $B \supseteq B_0$, and for any non-negative bounded function $f$ which is subharmonic in $B^c$
we have
\begin{align} \label{upper_bound_new}
f(x) \leq  C_2 \, \frac{1}{V(x)} \sum_{y \in B} P(x,y)f(y), \quad x \in B^c.
\end{align}
In particular,
$$
f(x) \leq C_2 \overline{K}_B  \frac{P(x,x_0)}{V(x)} \sum_{y \in B} f(y), \quad x \in B^c, \ x_0 \in B.
$$
\end{theorem}

\begin{remark}
The constant $C_2$ depends neither on $f$, $V$, nor on the set $B$.
\end{remark}

\begin{proof}[Proof of Theorem \ref{thm:main_1}]
The second assertion follows directly from the first one combined with Lemma \ref{lem:comparable}. We are left to show \eqref{upper_bound_new}. 

For any fixed $B \supseteq B_0$ we have
\begin{align}\label{split}
f (x) \leq \frac{1}{V(x)}\sum_{y\in B}P(x,y)f(y) + \frac{1}{V(x)}\sum_{y\in B^c} P(x,y) f(y), \quad x\in B^c.
\end{align}
Observe that \eqref{choice_r_0} implies $C_1C_\ast <1$. Hence 
\begin{align*}
f(x) \leq c_2 \sum_{y\in B}P(x,y)f(y) +C_1\left\|f\right\|_{\infty},\quad x\in B^c,
\end{align*}
where we set $c_2 = C_1(1 \vee C_\ast )<1$. 
This  estimate may be iterated with the aid of \eqref{split} and the DSP.
We claim that for any $n\in \mathbb{N}$
\begin{align}\label{iteration}
f(x) \leq (c_2+c_2^2\ldots +c_2^n)\sum_{y\in B}P(x,y)f(y) + C_1^n \left\|f\right\|_{\infty},\quad x\in B^c.	 
\end{align}
It suffices to prove the inductive step and for this we assume that \eqref{iteration} holds for any fixed $n\in \mathbb{N}$ and we show it for $n+1$. By using \eqref{iteration} to estimate $f(y)$ under the second sum in \eqref{split}, we obtain
\begin{align*}
f(x) &\leq c_2 \sum_{y\in B}P(x,y)f(y)\\
&\quad+ C_1(c_2+\ldots +c_2^n)\!\!\!\sum_{y \in B^c} \!\!\! P(x,y) \sum_{z\in B}P(y,z)f(z) + C_1^{n+1}\left\|f\right\|_{\infty},\quad x\in B^c.	 
\end{align*}
By applying Tonelli's theorem and the DSP to the double sum above, we get for $x\in B^c$
\begin{align*}
f(x)\leq c_2 \sum_{y\in B}P(x,y)f(y)+ C_1C_\ast (c_2+\ldots +c_2^n) \sum_{y\in B}P(x,y)f(y) + C_1^{n+1}\left\|f\right\|_{\infty},
\end{align*}
and the claim follows as $C_1C_\ast \leq c_2$. We next let $n$ to infinity in \eqref{iteration} and as the constants $C_1$ and $c_2$ were chosen to be smaller than one we arrive at
$$
f(x) \leq \frac{c_2}{1-c_2} \sum_{y \in B} P(x,y)f(y), \quad x\in B^c.
$$
Finally, by applying this inequality to estimate $f(y)$ under the second sum in \eqref{split} and using the DSP we conclude \eqref{upper_bound_new} with the constant $C_2:= 1+ (C_{\ast}c_2)/(1-c_2)$.
The proof is finished. 
\end{proof}

Next we show that the upper bound obtained in Theorem \ref{thm:main_1} is sharp in the sense that for all non-negative $(\calU -I)$-superharmonic functions we always have the matching lower bound. Note that for the lower bound we do not need assumption \eqref{B}.

\begin{proposition} \label{prop:lower_bound}
For any $D \subset X$, any non-negative function $f$ which is superharmonic in $D$, and  for any finite set $B \subset X$ we have
$$
f(x) \geq \frac{1}{V(x)} \sum_{y \in B} P(x,y) f(y), \quad x \in D.
$$
In particular, under assumption \eqref{A}, 
$$
f(x) \geq \underline{K}_B \frac{P(x,x_0)}{V(x)} \sum_{y \in B} f(y), \quad x \in D, \ x_0 \in B.
$$
\end{proposition}

\begin{proof} The first estimate follows directly from the inequality $(\calU -I) f(x) \leq 0$, $x \in D$. Indeed,
$$
f(x) \geq \frac{1}{V(x)}  \sum_{y \in X} P(x,y) f(y) \geq \frac{1}{V(x)}  \sum_{y \in B} P(x,y) f(y), \quad x \in D. 
$$
The second assertion is implied by Lemma \ref{lem:comparable}.
\end{proof}

\normal

The following important result is a consequence of Theorem \ref{thm:main_1} and Proposition \ref{prop:lower_bound}.

\begin{corollary}\label{cor:BHI_infty}
Under assumptions \eqref{A} and \eqref{B}, for any finite set $B \subset X$ with $B \supseteq B_0$, for any set $D \subset X$, and for any non-negative, non-zero and bounded function $f$ which is harmonic in $D$ and  such that $f(x) = 0$ for $x \in D^c \cap B^c$ we have
$$
\underline{K}_B \leq \frac{f(x)}{\frac{P(x,x_0)}{V(x)}  \sum_{y \in B} f(y)} \leq C_2\overline{K}_B, \quad x \in D \cap B^c, \ x_0 \in B,
$$ 
where $C_2$ is the constant of Theorem \ref{thm:main_1}.

In particular, the \textbf{uniform Boundary Harnack Inequality at infinity} holds: if $f$ and $g$ are two such non-zero harmonic functions, then 
$$
\left(\frac{\underline{K}_B}{C_2 \overline{K}_B}\right)^2 \leq \frac{f(x)g(y)}{g(x)f(y)} \leq  \left(\frac{C_2 \overline{K}_B}{\underline{K}_B}\right)^2 ,  \quad x, y \in D \cap B^c.
$$
\end{corollary} 

Finally, we remark that all the results proved in this section extend easily beyond the set-up of (sub-)probabilistic kernels. 

\begin{remark} \label{rem:ext} 
Theorem \ref{thm:main_1}, Proposition \ref{prop:lower_bound} and Corollary \ref{cor:BHI_infty} hold true for more general kernels $P:X \times X \to (0,\infty)$ which satisfy 
$$
M_1:=\sup_{x \in X} \sum_{y \in X} P(x,y) < \infty
$$ 
and 
$$
M_2:=\sup_{x, y \in X} \frac{\sum_{z \in X} P(x,z)P(z,y)}{P(x,y)} < \infty.
$$
Indeed, given such a kernel $P$ and a confining potential $V$ we can define $\widetilde P(x,y) := P(x,y)/M_1$ and $\widetilde V = V/M_1$ and observe that the sub-probability kernel $\widetilde P(x,y)$ satisfies assumption \eqref{A} with the constant $C_* = M_1 M_2$. Since 
$P(x,y)/V(x) = \widetilde P(x,y)/\widetilde V(x)$
and $\widetilde P$, $\widetilde V$ satisfy the assumptions of Theorem \ref{thm:main_1}, Proposition \ref{prop:lower_bound} and Corollary \ref{cor:BHI_infty}, all of these results apply to $P$ and $V$, and the dependence of a finite $B_0$ and all of the constants in the presented estimates remain unchanged. 
\end{remark}

\subsection{Estimates for nearest-neighbor random walks} \label{sec:est_nnrw}
In this paragraph we present a counterpart of the estimates obtained in Theorem \ref{thm:main_1} and Proposition \ref{prop:lower_bound} for the nearest-neighbor walk evolving in a graph. 

We start by imposing a graph structure in $X$. \normal The graph $G=(X,E)$ over $X$ (points in $X$ form the set of \emph{vertices}) is defined by specifying $E \subset \big\{\{x,y\}: x,y \in X \big\}$, the set of \emph{edges}. Two vertices $x, y \in X$ are connected by an edge in $G$ if and only if $\{x,y\} \in E$. In this case we call $x$ and $y$ \emph{neighbours} and write $x \sim y$ (note that $\{x,y\} = \{y,x\}$). We say that the graph $G$ %is \emph{simple} if $\{x,x\} \notin E$, for every $x \in X$ (i.e.\ there are no loops in $G$), and of 
is of \emph{finite geometry} if $\# \left\{y \in X: x \sim y \right\} < \infty$, for all $x \in X$ (i.e.\ the number of neighbours of an arbitrary vertex $x \in X$ is finite). Some authors call such a graph locally finite. Moreover, $G$ is \emph{connected} if for every $x,y \in X$, $x \neq y$, there exists a sequence $\big\{x_i\big\}_{i=0}^n \subset X$ with $x_0=x$, $x_n=y$ such that 
		$x_{i-1} \sim x_i,$ for $i = 1,\ldots,n$ (i.e.\ every two different vertices $x$ and $y$ are connected by a path in $G$). Every shortest path (the length of the path is counted as the number of edges belonging to that path) connecting two different vertices $x$ and $y$ is called a \emph{geodesic path} between $x$ and $y$.
For the rest of this section we assume that 
	\begin{itemize}
\item[\bfseries(\namedlabel{C}{C})] $G$ is a connected graph of finite geometry.  
\end{itemize}

\noindent
The assumption that $G$ is a connected graph allows us to define the natural \emph{graph} (\emph{geodesic}) \emph{distance} $d$ in $X$.
More precisely, $d(x,y)$ is defined as the length of the geodesic path connecting $x$ and $y$. As $G$ is of  finite geometry, every open geodesic ball $B_r(x) = \big\{y \in X: d(x,y) < r\big\}$ is finite and since $X$ is infinite, the metric space $(X,d)$ is unbounded.

We consider a  (sub-)probability kernel $P:X \times X \to [0,\infty)$ such that for every two vertices $x, y \in X$,
\begin{align} \label{eq:nnrw_def}
 P(x,y)>0 \quad \Longleftrightarrow \quad x \sim y .
\end{align} 
We do not assume that $P(x,y)$ is symmetric.
Let $\{S_n: n \geq 0\}$ be a time-homogeneous Markov chain  associated to $P$,
which is called a \emph{nearest-neighbor random walk} on a graph $G$. Due to the assumption of finite geometry the range of $S_n$ is a finite subset of $X$ for every $n$. This means that such a process can be understood as a discrete time counterpart of a diffusion in $X$.

The corresponding Feynman--Kac operator $\calU -I$ is given by
\begin{equation*}
(\calU -I)f (x)=  \frac{1}{V(x)} \sum_{y \in X} P(x,y) f(y) - f(x), \quad x \in X,
\end{equation*}
for all admissible functions $f$ on $X$.
To find satisfactory estimates of harmonic functions, we restrict our attention to the class of isotropic and increasing functions $V$. More precisely, we assume that there exists $x_0 \in X$ such that 
\begin{align*} %\label{eq:V_isotropic} 
V(x)=V(y),\ \text{if} \quad d(x,x_0) = d(y,x_0),\quad \mathrm{and}\quad 
V(x) \geq V(y),\ \text{if}\ d(x,x_0) \geq d(y,x_0).
\end{align*}
This can be equivalently stated as follows.
	\begin{itemize}  
\item[\bfseries(\namedlabel{D}{D})] There exist $x_0 \in X$ and an increasing function $W:\N_0 \to (0,\infty)$ such that $V(x) = W(d(x_0,x))$, for any $x \in X$.
\end{itemize}

\noindent
Our results apply well to the subclass of confining potentials that are isotropic and increasing, but formally we do not require assumption \eqref{B} in this paragraph. 
We first give the upper bound for $(\calU -I)$-subharmonic functions. 

\begin{theorem} \label{th:upper_subharm_nnrw} 
Let assumptions \eqref{C} and \eqref{D} hold with a fixed $x_0 \in X$ and a profile function $W$. Let $\calU -I$ be the Feynman--Kac operator corresponding to the  kernel $P(x,y)$ satisfying \eqref{eq:nnrw_def}. Then for any $r \in \N$ and for any non-negative and bounded function $f$ which is $(\calU -I)$-subharmonic in $ B_r(x_0)^c$ we have
\begin{equation*}
f(x)\leqslant \left\|f\right\|_{\infty} \prod_{i=r}^{d(x,x_0)} \frac{1}{W(i)}, \quad x \in B_r(x_0)^c.
\end{equation*}
\end{theorem}
	
\begin{proof}
Since $f$ is bounded and $(\calU -I)$-subharmonic in $B_r(x_0)^c$, 
\begin{equation}\label{first}
f(x)\leqslant \frac{1}{V(x)}\sum_{y \sim x}P(x,y)f(y)\leq \frac{1}{V(x)}\left\|f\right\|_{\infty}, \quad x \in B_r(x_0)^c.
\end{equation}
We next show that 
for any $j \geqslant 1$ and all $x \in X$ such that $d(x,x_0) \geqslant r+j$, and for any geodesic path $x_0\rightarrow \ldots\rightarrow x_n = x$ (clearly,  $n=d(x,x_0)$) it holds that 
\begin{equation}\label{Induction_to_show}
f(x) \leqslant \left\|f\right\|_{\infty} \big(V(x_n)V(x_{n-1})\ldots V(x_{n-j})\big)^{-1}.
\end{equation}
We use induction with respect to the parameter $j\geq 1$. If $j=1$ then for any  $x$ with 
 $d(x,x_0) \geq r+1$ and for any geodesic path $x_0\rightarrow \ldots\rightarrow x_n = x$ we apply \eqref{first} and we arrive at
\begin{align*}
f(x) \leqslant \frac{1}{V(x)}\left\|f\right\|_{\infty}\sum_{y\sim x}P(x,y)\frac{1}{V(y)}.
\end{align*}
Since $x=x_n$ and $x_{n-1}$ is one of its neighbours lying on a geodesic path connecting $x$ with $x_0$, we have $d(x_{n-1},x_0)\leqslant d(y,x_0)$, for every $y\sim x$. By \eqref{D} this implies that $V(x_{n-1}) \leq V(y)$, for $y\sim x$. In consequence, 
\begin{align*}
f(x)\leq \left\|f\right\|_{\infty} (V(x_{n}) \cdot V(x_{n-1}))^{-1},
\end{align*}
which proves \eqref{Induction_to_show} for $j=1$.

We proceed to the proof of the inductive step which will imply the desired  estimate. 
We assume that \eqref{Induction_to_show} is valid for some $j\geq 1$ and we aim to show that 
for any $x \in X$ with $d(x,x_0) \geqslant r+j+1$ and for any geodesic path $x_0\rightarrow \ldots\rightarrow x_n = x$ the following estimate is valid
\begin{equation*}
f(x) \leqslant \left\|f\right\|_{\infty} \big(V(x_{n}) V(x_{n-1}) \ldots V(x_{n-(j+1)})\big)^{-1}.
\end{equation*}
We fix $x$ with $d(x,x_0) \geqslant r+j+1$ and a geodesic path $x_0\rightarrow \ldots\rightarrow x_n = x$. 
By the $(\calU -I)$-subharmonicity, we have 
\begin{align}\label{split_111}
f(x) & \leqslant \frac{1}{V(x)}\sum_{z\sim x}P(x,z)f(z) \nonumber \\
& = \frac{1}{V(x)} \Big(\sum_{z\sim x,\atop  d(z,x_0)=n-1} + \sum_{z\sim x,\atop  d(z,x_0)=n} + \sum_{z\sim x,\atop  d(z,x_0)=n+1}\Big)P(x,z)f(z).
\end{align}
Since $V$ is isotropic,
$$
V(z_{n-1}) V(z_{n-2}) \ldots V(z_{n-1-j}) = V(x_{n-1}) V(x_{n-2}) \ldots V(x_{n-1-j}),
$$
for every $z \sim x$ such that $d(z,x_0)=n-1$, where $x_0\rightarrow z_1 \rightarrow \ldots\rightarrow z_{n-1} = z$ is the geodesic path connecting $x_0$ with $z$. In view of \eqref{Induction_to_show} it follows that 
\begin{equation*}
\sum_{z\sim x, \atop d(z,x_0)=n-1} P(x,z)f(z)
\leqslant\left\|f\right\|_{\infty}\big(V(x_{n-1}) \ldots V(x_{n-(j+1)})\big)^{-1} \sum_{z\sim x, \atop d(z,x_0)=n-1} P(x,z),
\end{equation*}
We next consider the second sum in \eqref{split_111}. Since $V$ is isotropic, for every $z \sim x$ such that $d(z,x_0)=n$, we have
$$
V(z_{n}) V(z_{n-1}) \ldots V(z_{n-j}) = V(x_{n}) V(x_{n-1}) \ldots V(x_{n-j}),
$$
where $x_0\rightarrow z_1 \rightarrow \ldots\rightarrow z_{n} = z$ is the geodesic path connecting $x_0$ with $z$. Therefore,
\begin{equation*}
\sum_{z\sim x, \atop d(z,x_0)=n}P(x,z)f(z)\leqslant \left\|f\right\|_{\infty} \big(V(x_n) V(x_{n-1}) \ldots V(x_{n-j})\big)^{-1} \sum_{z\sim x, \atop d(z,x_0)=n}P(x,z).
\end{equation*}
Since V is increasing, we also have
\begin{equation*}
\begin{split}
V(x_n) V(x_{n-1}) \ldots V(x_{n-j})
					\geq V(x_{n-1}) \ldots V(x_{n-(j+1)}).
\end{split}
\end{equation*}
By proceeding in a similar manner we find analogous upper bound for the third sum in \eqref{split_111}, that is
\begin{equation*}
\sum_{z\sim x, \atop d(z,x_0)=n+1}P(x,z)f(z)\leqslant \left\|f\right\|_{\infty}\big(V(x_{n-1}) \ldots V(x_{n-(j+1)})\big)^{-1}\sum_{z\sim x, \atop d(z,x_0)=n+1}P(x,z).
\end{equation*}
Now, by inserting all these bounds into \eqref{split_111}, we conclude that
\begin{align*}
f(x)\leqslant \left\|f\right\|_{\infty}\big(V(x_{n}) V(x_{n-1}) \ldots V(x_{n-(j+1)})\big)^{-1},
					\end{align*}
which finishes the proof of the inductive step.

Finally,  for any $x$ with $d(x,x_0) \geq r$,
\begin{equation*}
f(x) \leqslant \left\|f\right\|_{\infty}\big(V(x_{n}) V(x_{n-1}) \ldots V(x_r)\big)^{-1}
=\left\|f\right\|_{\infty} \prod_{i=r}^{d(x,x_0)}\frac{1}{W(i)},
\end{equation*}
which completes the proof.
\end{proof}

To obtain the lower bound for $(\calU -I)$-superharmonic functions we consider connected and geodesically convex subsets of $X$. The set $D \subset X$ is called \emph{geodesically convex} in a graph $G=(X,E)$ if $D$ contains each vertex on any geodesic path connecting vertices in $D$.
We also need an additional regularity assumption on the kernel $P(x,y)$, which coincides with the so-called $p_0$-condition imposed in \cite{Grigoryan-Telcs} (cf.\ \cite[Definition 2.1.1]{Kumagai}), that is
\begin{align} \label{eq:nnrw_add_ass} 
M:= \inf \left\{P(x,y): x,y \in X, \ x \sim y\right\} > 0.
\end{align}

\begin{theorem}  \label{th:lower_est_nnrw}
Let assumptions \eqref{C} and \eqref{D} hold with some $x_0 \in X$ and a profile function $W$. Let $\calU -I$ be the Feynman--Kac operator corresponding to the kernel $P(x,y)$ satisfying \eqref{eq:nnrw_def} and \eqref{eq:nnrw_add_ass}. Then, for any connected geodesically convex set $D \subset X$, for any non-negative function $f$ which is $(\calU -I)$-superharmonic in $D$, for any $x \in D$, and  for any $x_r \in D$ which lies on the geodesic path connecting $x$ with $x_0$ and is such that $d(x_r,x_0) = r < d(x,x_0)$, we have
\begin{equation}
f(x)\geq f(x_r) \prod_{i=r+1}^{d(x,x_0)}\frac{M}{W(i)}.
\end{equation}
\end{theorem}

\begin{proof}
We fix $x \in D$, a path $x_{0} \rightarrow\ldots\rightarrow x_n = x$ and $x_r \in D$.
By our assumptions, we have 
\begin{equation*}
f(x)\geqslant \frac{1}{V(x)} \sum_{y\in X}P(x,y)f(y)
\end{equation*}
and $x_{r+1},\ldots,x_{n-1} \in D$. 
Since $f$ is non-negative, we can write
\begin{equation*}
f(x)\geqslant  \frac{1}{V(x_n)} P(x_n,x_{n-1})f(x_{n-1})\geqslant \frac{M}{V(x_{n})} f(x_{n-1}).
\end{equation*}
Similarly, 
\begin{equation*}
f(x)\geqslant M e^{-V(x_n)}f(x_{n-1})\geqslant M^2 \exp\left(-V(x_n)-V(x_{n-1})\right)f(x_{n-2}).
\end{equation*}
By iterating this $(n-r)$--times, we arrive at  
\begin{equation*}
f(x)\geqslant M^{n-r} \big(V(x_n) V(x_{n-1} \ldots V(x_{r+1}\big)^{-1}f(x_{r}).
\end{equation*}
Since $V(x) = W(d(x_0,x))$ and $n=d(x,x_0)$, this leads to the desired bound. 
\end{proof}

To find the rate of decay at infinity of $(\calU -I)$-harmonic functions we impose an additional assumption on the profile function $W$, namely it is assumed that $\log W$ is regularly varying at infinity of index $\rho \geq 0$, see \cite{Bingham_book}.
We use the equality
$$
\prod_{i=r}^{d(x,x_0)}\frac{1}{W(i)} = \exp\left( -\sum_{i=r}^{d(x,x_0)} \log W(i) \right)
$$
and apply the following lemma to get a necessary estimate for the sum in the exponent.

\begin{lemma}{\cite[Lemma 2.4]{Nagaev}}\label{lem:reg_var}
Let $g$ be regularly varying at infinity of index $\rho\geq 0$. Then
\begin{align*}
\sum_{k=1}^ng(k) \sim \frac{ng(n)}{1+\rho}, \normal\quad \mathrm{as}\ n\to \infty.
\end{align*}
\end{lemma}
\noindent
The notation $a_n \sim b_n$ means that $\lim_{n \to \infty} \frac{a_n}{b_n} =1$. 

To simplify the statement we formulate the following corollary for the complement of a ball only. It, however, extends directly to more general unbounded sets $D \subset X$, cf.\ Theorem \ref{th:lower_est_nnrw}. 
\normal
\begin{corollary} \label{cor:nnrw_strongest}
Let assumptions \eqref{C} and \eqref{D} hold with a fixed $x_0 \in X$ and an increasing profile function $W$ such that $\log W$ is regularly varying at infinity of index $\rho \geq 0$. Let $\calU -I$ be the Feynman--Kac operator corresponding to the  kernel $P(x,y)$ satisfying \eqref{eq:nnrw_def} and \eqref{eq:nnrw_add_ass}. Then for any non-negative, non-zero and bounded function $f$ which is $(\calU -I)$-harmonic in $B_r(x_0)^c$ there are constants $C \geq 1$ and $\widetilde C >0$ such that
\begin{align*}
\frac{1}{C}  \exp\left(-\frac{1}{1+\rho}d(x,x_0) \right. & \left.\log W(d(x,x_0))  - \widetilde C d(x,x_0)  \right) \leq f(x)\\
&  \leq C\exp\left(-\frac{1}{1+\rho}d(x,x_0) \log W(d(x,x_0)) \right), \quad x \in B_r(x_0)^c.
\end{align*}
In particular,
$$
\lim_{ d(x,x_0) \to \infty} \frac{\log f(x)}{d(x,x_0) \log W(d(x,x_0))} = -\frac{1}{1+\rho}.
$$
\end{corollary}

\begin{remark} \label{rem:ext_nnrw} 
Theorem \ref{th:upper_subharm_nnrw} can also be easily extended to a more general setting by considering $\widetilde P(x,y) := P(x,y)/M_1$ in the case when 
$$
M_1:=\sup_{x \in X} \sum_{y \in X} P(x,y) \in (1,\infty).
$$ 
This would lead to the upper bound of the form
 \begin{equation*}
f(x)\leqslant \left\|f\right\|_{\infty} \prod_{i=r+1}^{d(x,x_0)}\frac{M_1}{W(i)}, \quad x \in B_r(x_0)^c.
\end{equation*}
\end{remark}

\section{Applications} \label{sec4}

We provide a few applications of the presented estimates of functions which are  harmonic with respect to the Feynman-Kac operators.

\subsection{Decay of solutions to equations involving the graph Laplacians} \label{sec:graph_L} 
 By following the line of the papers dealing with graph Laplacians (see, e.g.\ \cite{Keller_Lenz} and further references in the monograph \cite{keller_lenz_wojciechowski_2021}), we impose the structure of a weighted graph over a given countable infinite space $X$ by considering a  kernel $b : X \times X \to [0,\infty]$ such that 
\begin{itemize}
\item[(i)] $b(x,x)=0$, for every $x \in X$;
\item[(ii)] $b(x,y)=b(y,x)$, for every $x,y \in X$;
\item[(iii)] $\sum_{y \in X} b(x,y)>0$, for every $x \in X$, and $\sup_{x \in X} \sum_{y \in X} b(x,y)<\infty$. 
\end{itemize}
Let $m : X \to (0,\infty)$ be a (strictly positive) measure on $X$. 
We additionally consider a function $V : X \to \R$ such that $\inf_{x\in X} V(x) > - \infty$. The \emph{graph Laplacian} $H$ is defined by
$$
H f (x) = \frac{1}{m(x)} \sum_{y \in X} b(x,y) \big(f(x)-f(y)\big) + V(x) f(x),
$$
for all functions $f \in F:=\{f:X \to \R: \sum_{y} b(x,y) |f(y)| < \infty, \ \text{for every} \ x \in X \}$.
The triple $(X,b,V)$ can be seen as a weighted graph over $X$ (two points $x,y \in X$ form an edge if and only if $b(x,y) >0$).

We first establish the relation between the operator $H$ and the discrete Feynman--Kac operator $\calU -I$. We set
\begin{align*}
b(x) = \sum_{y \in X} b(x,y), \qquad b^*:= \sup_{x \in X} b(x), \qquad P(x,y) = \frac{b(x,y)}{b^*},
\end{align*}
and
\begin{align*}
A = \left\{x \in X: m(x)V(x) + b(x) \leq 0 \right\},\quad
\widetilde V(x) = \begin{cases}
\frac{m(x)V(x)+b(x)}{b^*}, \quad  x \in A^c, \\
0 , \quad \ \ \ \ \ \ \ \ \ \ \ \ \ \ \ x \in A.
\end{cases} \normal
\end{align*}

We further assume that the operator $\calU$ is defined as in \eqref{def:calU} with a sub-probability kernel $P(x,y)$ and the potential $\widetilde V(x)$ defined  above. \

\begin{proposition} \label{prop:H_and_A} For every $f \in F$ and $x \in A^c$ we have
\begin{align} \label{eg:H_and_A}
H f(x) = - \left(V(x)+\frac{b(x)}{m(x)}\right) \ (\calU -I) f(x).
\end{align}
In particular, if $D \subset A^c$ and $f \in F$, then
$$
H f(x) \geq 0, \quad x \in D \qquad \Longleftrightarrow \qquad  (\calU -I) f(x) \leq 0, \quad x \in D.
$$
\end{proposition}

\begin{proof}
The proof of the first equality is based on a direct rearrangement. Indeed, for $x \in A^c$, 
\begin{align*} 
H f (x) & = \frac{1}{m(x)} \sum_{y \in X} b(x,y) \big(f(x)-f(y)\big) + V(x) f(x)\\
        & = \left(V(x)+\frac{b(x)}{m(x)}\right)  f(x) - \frac{b^*}{m(x)} \sum_{y \in X} \frac{b(x,y)}{b^*} f(y) \\
				& = \left(V(x)+\frac{b(x)}{m(x)}\right) \left(f(x) - \frac{b^*}{m(x)V(x) + b(x)} \sum_{y \in X} \frac{b(x,y)}{b^*} f(y)\right) \\
				& = \left(V(x)+\frac{b(x)}{m(x)}\right)  \left(f(x) - \frac{b^*}{m(x)V(x) + b(x)} \sum_{y \in X} \frac{b(x,y)}{b^*} f(y)\right).
\end{align*}
This is exactly \eqref{eg:H_and_A}. The second assertion follows directly from \eqref{eg:H_and_A} and the definition of the set $A$.
\end{proof}

It follows that Theorem \ref{thm:main_1} and Theorem \ref{th:upper_subharm_nnrw} can be effectively used to get an upper bound for subsolutions to the equation $H f =0$ (i.e.\ for $f$ such that $Hf \leq 0$) outside of a finite set.
Moreover, Proposition \ref{prop:H_and_A} implies that $H$ and $\calU$ have the same harmonic functions in subsets of $A^c$ (to see this we apply Proposition \ref{prop:H_and_A} to $f$ and $-f$).
This simple observation allows one to reduce the study of properties of functions harmonic with respect to $H$ to those which are harmonic with respect to the Feynman--Kac operators.
By combining this fact with results of Section \ref{sec:gen_est}, we obtain the following estimates for the solutions to the equation $H f =0$ in infinite sebsets of $A^c$. 

We note that when $m(x)V(x)+b(x)$ is positive for every $x \in X$ (that is, $A^c = X$), then \eqref{eg:H_and_A} can be seen as a standard change of measure procedure. 
A similar approach has been used recently by Fischer and Keller in \cite[p.\ 16]{Fischer_Keller}.

Observe that if $V$ is a confining potential in the sense of \eqref{B} and $\inf_{x \in X} m(x) >0$, then also $\widetilde V(x)$ is a confining potential and $A$ is at most finite.

\begin{corollary} \label{thm:main_H}\textup{\textbf{(DSP case)}}
Suppose that $b(x,y)$, $m(x)$ and $V(x)$ are as above. Assume that $V$ satisfies \eqref{B}, $\inf_{x \in X} m(x) >0$ and that
$$
b(x,y) > 0, \quad \mathrm{and}\quad  \sup_{x,y \in X} \sum_{z \in X}\frac{b(x,z)b(z,y)}{b(x,y)} < \infty. 
$$
Let $D \subset X$ and let $f$ be a bounded solution to the equation $H f (x) = 0$, $x \in D$. Then the following assertions hold. 
\begin{itemize}
\item[(1)]  There exists a finite set $B_0 \subset X$ (independent of $m$, $D$ and $f$) with $B_0 \supseteq A$ such that for any finite set $B \subset X$ with $B \supseteq B_0$ there exists a constant $C >0$ (independent of $V$, $m$, $D$ and $f$) such that
$$
|f(x)| \leq C \frac{b(x,x_0)}{m(x)V(x) + b(x)}\sum_{y \in B}|f(y)|, \quad x \in D \cap B^c, \ x_0 \in B,
$$
whenever $f(x) = 0$ for $x \in D^c \cap B^c$;
\item[(2)] If, in addition, $f$ is non-negative, then for any finite set $B \subset X$ with $B \supseteq B_0$ there exists a constant $\widetilde C >0$ (independent of $V$, $m$, $D$ and $f$) such that
$$
f(x) \geq \widetilde C  \frac{b(x,x_0)}{m(x)V(x) + b(x)}\sum_{y \in B}f(y), \quad x \in D \cap B^c, \  x_0 \in B.
$$
In particular, the \textbf{uniform Boundary Harnack Inequality at infinity} holds (cf.\ Corollary \ref{cor:BHI_infty}).
\end{itemize}
\end{corollary}

\begin{proof} 
First note that by Proposition \ref{prop:H_and_A} the function $f$ is $(\calU -I)$-harmonic in $D \cap A^c$. 
To justify the upper bound in (1) it is then enough to observe that $|f|$ is $(\calU -I)$-subharmonic in $B^c$ and apply Theorem \ref{thm:main_1}. The corresponding lower bound (2) follows from Proposition \ref{prop:lower_bound}. Finally, the sharp two-sided estimates lead to the uBHP at infinity as in Corollary \ref{cor:BHI_infty}.
\end{proof}

For simplicity we formulate the following result under the assumption that the functions $m(x)$ and $b(x)$ are constant. It is, however, not difficult to derive similar estimates for the case when $0 < \inf_{x \in X} m(x) \leq \sup_{x \in X} m(x) < \infty$ and $0 < \inf_{x \in X} b(x) \leq \sup_{x \in X} b(x) < \infty$. 

\begin{corollary} \label{thm:main_H_nnrw} \textup{\textbf{(Nearest-neighbour case)}}
Let $G = (X,E)$ be a graph such that assumption (\ref{C}) holds. Denote by $d(x,y)$ the (geodesic) graph distance in $X$. Suppose that $b(x,y)$ satisfy \eqref{eq:nnrw_def}, $m(x)$ is as above, and $V(x)$ satisfies assumption (\ref{D}) with some $x_0 \in X$ and a profile $W$. Moreover, assume that there are positive numbers $b_0, m_0$ such that $b(x)= b_0$ and $m(x) = m_0$ for all $x\in X$. 
Then the following assertions hold.
\begin{itemize}
\item[(1)] If $D \subset X$, $r \in \N$ is such that $A \subset B_{r}(x_0)$, and $f$ is a bounded solution to the equation $H f (x) = 0$, $x \in D$, such that $f(x) = 0$, $x \in D^c \cap B_{r}(x_0)^c$, then
$$
|f(x)| \leq \left\|f\right\|_{\infty} \prod_{i = r+1}^{d(x,x_0)} \frac{b_0}{b_0+m_0W(i)}, \quad x \in D \cap B_{r}(x_0)^c.
$$ 
\item[(2)] If, in addition, the kernel $b(x,y)$ satisfies \eqref{eq:nnrw_add_ass},
$f$ is non-negative and $D$ is geodesically convex, then there exists $C > 0$ (independent of $V$, $m$, $D$ and $f$) such that for every $x_r \in D$ with $d(x_0,x_r) = r \in \N$ we have
$$
f(x) \geq f(x_r) \prod_{i = r+1}^{d(x,x_0)} \frac{C}{b_0+m_0W(i)}, \quad x \in D \cap B_{r}(x_0)^c.
$$
\end{itemize}
\end{corollary}

\begin{remark}
When $b(x,y)$ is a probability kernel and $m$ is a counting measure (i.e.\ $m \equiv 1$), then $b_0 = m_0 = 1$ and the estimates in Corollary \ref{thm:main_H_nnrw} simplify and become sharper. 
\end{remark}

\begin{proof}[Proof of Corollary \ref{thm:main_H_nnrw}] 
By Proposition \ref{prop:H_and_A}, the function $f$ is $(\calU -I)$-harmonic in $D \cap A^c$. 
It follows that the function $|f|$ is $(\calU -I)$-subharmonic in $B_r(x_0)^c$.
The potential $\widetilde V$ satisfies assumption \eqref{D} with the profile $(m_0W + b_0)/b_0$ and the same $x_0 \in X$. Therefore we get the upper bound for $|f|$ by employing Theorem \ref{th:upper_subharm_nnrw}.

The lower estimate in (2) follows from Theorem \ref{th:lower_est_nnrw} by a similar argument. 
\end{proof}

\begin{remark} \label{rem:boundedness}
In Corollary \ref{thm:main_H} and Corollary \ref{thm:main_H_nnrw} (and in the results of Section \ref{sec:gen_est}) we assume that the function $f$ is bounded. This assumption is not  restrictive in the context of applications to graph Laplacians. For example, it is evident that if the measure $m$ satisfies the condition $\inf_{x \in X} m(x) > 0$ then every function $f \in \ell^p(X,m)$, $1 \leq p < \infty$ is bounded. 
\end{remark}

\begin{remark}
In the above corollaries we did not assume that $b(x,y)$ is symmetric.
If, however, $b(x,y)$ is symmetric, $\inf_{x \in X} m(x) >0$ and $V(x)$ is a confining potential, then $H$ is an  unbounded self-adjoint operator on $\ell^2(X,m)$ with the dense domain $D(H) = \{f \in \ell^2(X,m): Hf \in \ell^2(X,m)\}$ \cite{Keller_Lenz}. In this particular case, the obtained results seem to be of special interest as such operators serve as \emph{Hamiltonians} in \emph{discrete models of quantum oscillators} \cite{Chalbaud_1986, Gallinar_Chalbaud, Mattis}.
Since $H$ is self-adjoint, its spectrum is real and it consists of a countable set of eigenvalues of finite multiplicities; this sequence has no limit points and diverges to infinity. Eigenfunctions and
eigenvalues of the Schr\"odinger operator $H$  are called \emph{energy eigenstates} and \emph{energy levels} of the system. The eigenvalue $\lambda_0$ which lies at the bottom of the spectrum of $H$ has multiplicity 1; it describes the energy of the quantum system in the so-called \emph{ground state}. The respective eigenfunction $\psi_0 \in \ell^2(X,m)$ is positive. In general, if $\psi$ is a normalized eigenfunction of the operator $H$, then $|\psi(x)|^2$ is the density of the probability distribution of the position of a particle in a quantum state respective to $\psi$. Therefore, the knowledge of the fall-off rates of $\psi$ provides an information about the localization of the particle in a configuration space. 
\end{remark}

We finally  show how one can apply our results to describe the decay rate of eigenfunctions of operator $H$ outside of a finite set. 
Let $V$ be a confining potential and let $\psi \in \ell^2(X,m)$ be an eigenfunction of the operator $H$ corresponding to the eigenvalue $\lambda \in \mathbb R$. It follows that
\begin{equation} \label{osc}
H_{\lambda} \psi = 0, 
\end{equation}
where
$$
H_{\lambda} f (x) = \frac{1}{m(x)} \sum_{y \in X} b(x,y) \big(f(x)-f(y)\big) + V_{\lambda}(x) f(x), \quad \text{with} \quad V_{\lambda}(x):=V(x)-\lambda.
$$
Note that under the assumption that $\inf_{x \in X} m(x) >0$ the eigenfunction $\psi$ is bounded on $X$ (see Remark \ref{rem:boundedness}).
Take $D = A_{\lambda}^c$, where
$$
A_{\lambda}=\left\{x \in X: V_\lambda(x) \leq - \frac{b(x)}{m(x)}\right\}.
$$
In this framework one can directly apply Corollary \ref{thm:main_H} and Corollary \ref{thm:main_H_nnrw} to obtain the upper bound for $|\psi(x)|$ outside of a finite set in the DSP and the nearest-neighbour case. For the ground state eigenfunction, i.e.\ $\lambda=\lambda_0$ and $\psi = \psi_0$, we also obtain the matching lower bound. It is remarkable that in the DSP case the decay of any eigenfunction of $H$ at infinity is dominated by that of the ground state eigenfunction $\psi_0$. 
Some concrete examples will be discussed in Section \ref{sec:expl_rates}.

\subsection{Estimates for eigenfunctions of discrete Feynman--Kac operators} \label{sec:est_ef_FK}

Estimates obtained in Section \ref{sec:gen_est} can be effectively used to investigate the spectral and analytic properties of discrete time Feynman-Kac semigroups. Recall that the semigroup $\{\calU_n: n \in \N_0\}$ consists of operators
$\calU_0 f = f $, $\calU_n f = \calU^{n} f$, for $n\geq 1$,
where
$$
\calU f(x) = \frac{1}{V(x)} \sum_{y \in X} P(x,y)f(y), \quad x \in X.
$$
Suppose we are given a positive measure $\mu$ on $X$ such that 
$$
\mathrm{(i)\ }\sup_{y \in X} \frac{\sum_{x \in X} \mu(x) P(x,y)}{\mu(y)} < \infty,\quad \mathrm{and}\qquad \mathrm{(ii)}\ \sup_{x,y \in X} \frac{P(x,y)}{\mu(y)} < \infty.
$$
Under condition (i), the operator $\calU$ is bounded in $\ell^p(X,\mu)$, for any $1 \leq p < \infty$ 
(observe that the $\calU$ is also bounded in $\ell^{\infty}(X,\mu)$ as $P(x,y)$ is a probability kernel and $\mu$ is a positive measure on $X$). Condition (ii) implies that the operator $\calU: \ell^p(X,\mu) \to \ell^{\infty}(X,\mu)$ is bounded for every $1 \leq p < \infty$. 

From now on we restrict our attention to the case of $\ell^2(X,\mu)$. We first show that under \eqref{B} the operator $\calU$ is compact in $\ell^2(X,\mu)$. Clearly, this property is inherited by all the semigroup operators $\calU_n$, $n \geq 1$. The following lemma seems to be a standard fact, but we give a short proof for reader's convenience.  

\begin{lemma} \label{lem:cpt}
Under assumption \eqref{B}, the operator $\calU$ is compact in $\ell^2(X,\mu)$.
\end{lemma}
\begin{proof}
We define the following sequence of finite-rank operators
\begin{align*}
\calU^{(k)} f(x) = \mathbf{1}_{B_k}(x)\frac{1}{V(x)}\sum_{y} P(x,y) f(y), \quad k \in \N,
\end{align*}
where $\left\{B_k\right\}_{k \in \N}$ is a family of finite subsets of $X$ such that $V(x) \geq k$ for $x \in B_k^c$, see \eqref{B}. 
We aim to prove that $\calU^{(k)} $ converges to $\calU$ in the operator norm. This will imply the desired compactness of $\calU$ as any finite-rank operator is compact. The Cauchy--Schwarz inequality combined with Tonelli's theorem yield
\begin{align*}
\Vert (\calU-\calU^{(k)})f\Vert_2^2 &= \sum_{x\notin B_k}\left(\frac{1}{V(x)}\right)^2 \Big| \sum_y f(y) P(x,y)\Big|^2\mu (x) \\
&\leq \left(\frac{1}{\inf_{x\notin B_k}V(x)}\right)^2 \sum_y  \sum_{x}\frac{\mu(x)P(x,y)}{\mu(y)}|f(y)|^2 \mu(y)\\
 & \leq \frac{1}{k^2} \left(\sup_{y \in X} \frac{\sum_{x} \mu(x)P(x,y)}{\mu(y)} \right) \Vert f\Vert_2^2.
\end{align*}
By (i), the last expression converges to zero as $k \to \infty$ and the result follows.
\end{proof}

\begin{remark} \label{rem:dual_disc}
We do not assume that the kernel $P(x,y)$ is symmetric. In consequence, the operator $\calU$ need not be self-adjoint in $\ell^2(X,\mu)$. The duality issue will be discussed in Section \ref{sec:adjoint_F-K}. 
\end{remark}

We deduce that the spectrum of the operator $\calU$ (excluding zero) consists solely of eigenvalues.  Moreover, by Jentzsch theorem \cite[Theorem V.6.6.]{Schaefer}, the spectral radius of $\calU$ is an eigenvalue, which we denote by $\lambda_0>0$, and the corresponding  eigenfunction $\psi_0$ is strictly positive on $X$. 

Let $\lambda \in \mathbb C$, $\lambda \neq 0$ be an eigenvalue of the operator $\calU$ and let $\psi \in \ell^2(X,\mu)$ be the corresponding eigenfunction, i.e. 
$
\calU \psi = \lambda \psi.
$
We then have
$
|\lambda| |\psi| = |\calU \psi| \leq \calU |\psi|
$, which implies $|\psi| \leq \calU^{\lambda} |\psi|$,
where
$$
\calU^{\lambda} f(x) = \frac{1}{V_{\lambda}(x)} \sum_{y \in X} P(x,y) |\psi(y)|, \quad \text{with} \ V_{\lambda}:= |\lambda|V.
$$
In particular, 
$(\calU^{\lambda} - I) |\psi|(x) \geq 0$, $x \in X$,
i.e.\ the non-negative function $\varphi:= |\psi|$ is $(\calU^{\lambda} - I)$-subharmonic in $X$. We show similarly that the positive function $\psi_0$ is $(\calU^{\lambda} - I)$-harmonic.

After this preparation we can apply Theorem \ref{thm:main_1} and Theorem \ref{th:upper_subharm_nnrw}  to obtain an upper bound for $|\psi|$ outside of a finite set in the DSP and the nearest-neighbour case, respectively. By Proposition \ref{prop:lower_bound} and Theorem \ref{th:lower_est_nnrw}, we can also find the matching lower bound for the positive eigenfunction $\psi_0$ in this two cases. 

Our main contribution here is that we can find sharp two-sided bounds for $\psi_0$ outside of a finite set. As we mentioned in the introduction, we apply this result in our ongoing work  to investigate the asymptotic behaviour of the kernel of the operator $\calU_n$.

\subsection{Conjugate Feynman--Kac semigroups and the duality issue} \label{sec:adjoint_F-K}

We close this section with a short discussion concerning the definition of the discrete Feynman--Kac operators and the duality issue.   

As we mentioned in the introduction, for a given sub-probability kernel $P$ and a potential $V$ the corresponding discrete time Feynman--Kac semigroup can be defined in an alternative way. More precisely, we consider a semigroup $\{\calW_n: n \in \N_0\}$ consisting of operators given by
\begin{align} \label{eq:mult_fun_W}
\calW_0 g = g, \quad \calW_n g(x) = \ex^x \bigg[\prod_{k=1}^{n}\frac{1}{V(Y_k)} g(Y_n) \bigg], \quad n\geq 1
\end{align}
We have 
$\calW_n g = \calW^{n} g$, for $n\geq 1$, 
where $\calW^{n}$ denotes the $n$\textsuperscript{th} power of the operator
\begin{align*}
\calW g(x) =  \sum_{y \in X} P(x,y)(g(y)/V(y)), \quad x \in X.
\end{align*}
Observe that 
\begin{align} \label{eq:U-to-W}
V^{-1}\calW g =  \calU (V^{-1} g),
\end{align}
for all admissible functions $g$. We call $\{\calW_n: n \in \N_0\}$ the \emph{conjugate discrete time Feynman--Kac semigroup} to $\{\calU_n: n \in \N_0\}$.

\begin{remark}
\begin{enumerate}
\item[(1)] In view of identity \eqref{eq:U-to-W}
$g$ is $(\calW -I)$-harmonic (superharmonic, subharmonic) in $D$
if and only if 
$g/V$ is $(\calU -I)$-harmonic (superharmonic, subharmonic)  in $D$.
This allows us to apply all results obtained in Section \ref{sec:gen_est} to the operator $\calW -I$.

\item[(2)] If condition (i) in Section \ref{sec:est_ef_FK} is satisfied and the following reversibility relation
$$
\mu(x) P(x,y) = \mu(y)P(y,x), \quad x,y \in X,
$$
holds, then the operator $\calW$ is adjoint to $\calU$ in $\ell^2(X,\mu)$. Indeed, for every $f,g \in \ell^2(X,\mu)$, we have
\begin{align*}
\sum_{x \in X} g(x) \overline{\calU f(x)} \mu(x) & = \sum_{x \in X} \sum_{y \in X} \frac{g(x)}{V(x)} \mu(x)P(x,y) \overline{f(y)} \\
                                      & = \sum_{y \in X} \sum_{x \in X}  \frac{1}{V(x)} P(y,x) g(x)\overline{f(y)}  \mu(y) \\
																			& = \sum_{y \in X} \calW g(y) \overline{f(y)} \mu(y),
\end{align*}
by Fubini's theorem. Clearly, this extends to the  operators $\calU_n$ and $\calW_n$, $n \geq 1$.
We also observe that $\calU$ acts as a self-adjoint operator in the space $\ell^2(X,\mu_V)$, where $\mu_V(x) = V(x)\mu(x)$. Indeed,
for $f,g \in \ell^2(X,\mu_V)$ we have
\textbf{\begin{align*}
\sum_{x \in X} g(x) \overline{\calU f(x)} \mu_V(x) & = \sum_{x \in X} \sum_{y \in X} g(x) \mu(x)P(x,y) \overline{f(y)} \\
                                      & = \sum_{y \in X} \sum_{x \in X} g(x) \mu(y)P(y,x) \overline{f(y)}  \\
																			& = \sum_{y \in X} \sum_{x \in X}  \frac{1}{V(y)} P(y,x)g(x) \overline{f(y)}  \mu_V(y) \\
																			& = \sum_{y \in X} \calU g(y) \overline{f(y)} \mu_V(y).
\end{align*}}
\end{enumerate}
\end{remark}

\section{Markov chains with the DSP}\label{sec:DSP_chains}
In this section we present various methods which allow one to construct (sub-)probability kernels that satisfy the direct step property. We start with a few general examples on metric spaces. Next, we give more precise results for a class of discrete-time processes (constructed through subordination techniques) on infinite countable sets, including weighted graphs, and also discuss a class of chains defined on product spaces. Finally, we give a  few direct examples where we evaluate the decay rates which appear in the obtained estimates of harmonic functions.

\subsection{(Sub-)Markov kernels with the DSP on metric spaces} \label{sec:chains_metric}
Let $(X,d)$ be a countable metric space and let $P(x,y)$ be a given (sub-)probability kernel on $X$. In the first result we show that if the kernel depends on the distance through a function that satisfies an appropriate doubling condition then such a kernel fulfils assumption \eqref{A}.

\begin{proposition}\label{prop:doubling}
Let $P(x,y)$ be a (sub-)probability kernel such that
\begin{align}\label{p-J-condition}
P(x,y)\asymp J(d(x,y)),\quad x,y\in X,
\end{align}
for a non-increasing function $J:[0,\infty)\to (0,\infty)$ which satisfies the doubling condition
\begin{align}\label{doubling_J}
J(r)\leq C J(2r),\quad r>0.
\end{align}
Then the kernel $P(x,y)$ satisfies assumption \eqref{A}.
\end{proposition}

\begin{proof}
The kernel $P(x,y)$ is strictly positive by \eqref{p-J-condition}, so we only need to prove \eqref{eq:DSP}. We aim to show that there is a constant $c>0$ such that for all $x,y\in X$
\begin{align*}
\sum_{z\in X} J(d(x,z))J(d(z,y))\leq c J(d(x,y)).
\end{align*}
We split the sum according to the distance of $z$ to $x$ and $y$.
For $d(x,z)>d(x,y)$ we have by monotonicity that $J(d(x,z))\leq J(d(x,y))$ and thus \eqref{def:prob_kernel} combined with \eqref{p-J-condition} imply
\begin{align*}
\sum_{\left\{z:\, d(x,z)>d(x,y)\right\}} \! J(d(x,z))J(d(z,y))\leq c J(d(x,y)).
\end{align*}
We proceed similarly for $d(y,z)>d(x,y)$. If $\max\{d(x,z), d(z,y)\}\leq d(x,y)$ (this coincides with the shadowed area in Figure \ref{fig:balls}), we distinguish between two cases: either $d(x,z)\geq \frac{d(x,y)}{2}$ (region $II$ in Fig.\ \ref{fig:balls}), then $J(d(x,z))\leq J(\frac{d(x,y)}{2})\leq C J(d(x,y))$ by \eqref{doubling_J}; or $d(x,z)<\frac{d(x,y)}{2}$ (region $I$), then $d(z,y)\geq \frac{d(x,y)}{2}$ and we have $J(d(z,y))\leq J(\frac{d(x,y)}{2})\leq CJ(d(x,y))$. Hence, by employing \eqref{def:prob_kernel} and \eqref{p-J-condition} to each of the two cases we obtain the desired result.
\end{proof}
\begin{figure}[h]
\centering
\begin{tikzpicture}[scale=1.5]
\draw [black] (1,0) rectangle (7,3);
\path[fill=gray, opacity=0.5, postaction={fill=white, opacity=1, even odd rule, draw}] (3.5, 1.5) circle (1) (4.5, 1.5) circle (1);
\node at (3.3,1.5) {$x$};
\node at (4.7, 1.5) {$y$};
\draw [black]  (3.5,1.5) circle (1pt);
\draw [black]  (4.5,1.5) circle (1pt);

% find the intersection of two circles
\tkzDefPoint(3.5,1.5){A}
\tkzDefPoint(4.5,1.5){B}
%\tkzDrawSegment(A,B)

\tkzInterCC(A,B)(B,A) \tkzGetPoints{C}{X}
% seperating line in the iteresection
\tkzDrawSegment[dashed](C,X)

\node at (6.5, 0.2) {$X$};   %space name
% Regions names
\node at (3.7, 1.7) {$I$};
\node at (4.2, 1.3) {$II$};
%\begin{scope}
%\clip (-1,0) to [out=90, in=310] (-3,7)-- (3,7) to [out=230, in=90] 		(1,0)--(-1,0);
%\clip (0:3cm) -- (0:6cm) arc (0:180:6cm) -- (180:3cm)
%										arc (180:0:3cm) -- cycle;
%										\fill[color=blue, opacity=0.3] (-5,0) rectangle (5,7); 
%									\end{scope}
%\draw [gray] (6,0) arc [radius=1, start angle=45, end angle= 120];

%\draw(0,0) circle (1cm);
%\draw(1,0) circle (1cm);
%\draw [clip](0,0) circle (1cm);
%\fill[red] (1,0) circle (1cm);

\end{tikzpicture}
\caption{Two intersecting balls of radius $d(x,y)$.}
\label{fig:balls}
\end{figure}
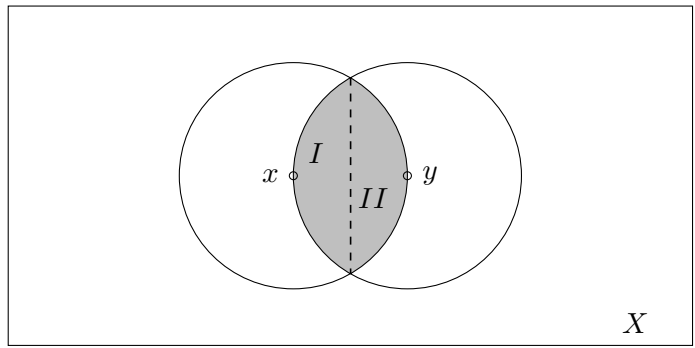

We remark that Proposition \ref{prop:doubling} is applicable to plenty of long range random walks. Prominent examples are random walks on the integer lattice with one-step transition kernel defined through a family of conductances that are comparable to the jump kernels of stable processes, see \cite{Bass_Levin}. Another important examples are random walks on uniformly discrete metric measure spaces studied in \cite{Murugan_Saloff-Coste} (see also \cite{Murugan_Saloff_2}) where the jump kernel is comparable to a regularly varying function (depending on the distance function) times a volume growth function which satisfy a doubling condition.

We also give a natural extension of the above result which covers many interesting examples of kernels $P(x,y)$ that decay faster at infinity than those in Proposition \ref{prop:doubling}. 

\begin{corollary} \label{cor:JandK}
Let $P(x,y)$ be a (sub-)probability kernel such that
\begin{align}\label{P-J-K-condition}
P(x,y)\asymp J(d(x,y)) K(d(x,y)),\quad x,y\in X,
\end{align}
where $J, K:[0,\infty)\to (0,\infty)$ are non-increasing functions such that 
$J$ satisfies \eqref{doubling_J} and $K$ is such that
\begin{align}\label{K-condition}
K(r)K(s) \leq \widetilde C K(r+s), \quad r,s >0.
\end{align}
Then the kernel $P(x,y)$ satisfies assumption \eqref{A}.
\end{corollary}

\begin{proof}
We only need to show \eqref{eq:DSP}. By \eqref{K-condition} and using the monotonicity of the function $K$ and the triangle inequality we obtain
$$
K(d(x,z))K(d(z,y)) \leq \widetilde C K(d(x,z)+d(z,y)) \leq \widetilde C K(d(x,y)), \quad x,y,z \in X.
$$
It then follows from Proposition \ref{prop:doubling} that for all $x,y\in X$,
\begin{align*}
\sum_{z\in X} K(d(x,z))&J(d(x,z))K(d(z,y))J(d(z,y)) \\
& \leq \widetilde C K(d(x,y)) \sum_{z\in X} J(d(x,z))J(d(z,y))\leq c K(d(x,y)) J(d(x,y)).
\end{align*}
Together with \eqref{P-J-K-condition} this implies \eqref{eq:DSP}. 
\end{proof}
Typical examples of profiles $J$ and $K$ satisfying the assumptions of Proposition \ref{prop:doubling} and Corollary \ref{cor:JandK} are: $J(r) = (1 \vee r)^{-\gamma}$, or $J(r) = (1 \vee r)^{-\gamma} \log(2+r)^{\delta}$, for appropriate $\gamma >0$ and $\delta \in \R$ (which depend on the geometry of the space $X$), and $K(r) = e^{-\theta r^{\beta}}$, $\theta>0$, $\beta \in (0,1]$.  

\subsection{Subordinate Markov chains} \label{sec:chains_subordinate}
In this paragraph we consider a specific class of Markov chains which enjoy the DSP property and are obtained via a random change of time procedure. 
We start with a straightforward but at the same time very fruitful observation that the DSP property is stable under a random change of time. For this, let $\{Z_n: n \geq 0\}$ be an arbitrary time-homogeneous Markov chain with values in $X$ and 
let $\{\tau_n:n \geq 0 \}$ be an arbitrary increasing random walk starting at 0 with values in $\N_0$ and which is independent of $\{Z_n: n \geq 0\}$ (by saying that it is a random walk we mean that $\tau_{n+1}-\tau_n$, $n=0,1,2.\ldots$ are i.i.d.\ random variables). 
The \emph{subordinate Markov chain} $\{Y_n: n \geq 0\}$
is then defined as 
$$Y_n:= Z_{\tau_n}, \quad n=0,1,2,\ldots . $$ 
It is straightforward to check that the process $\{Y_n: n \geq 0\}$ is indeed a time-homogeneous Markov chain.  

\begin{lemma}\label{lem:subord} 
 Suppose that $\{\tau_n:n \geq 0 \}$ satisfies the DSP, \normal that is 
\begin{align}\label{eq:DSP_N} 
\bbP(\tau_2 = n)  \leq C \bbP(\tau_1 = n), \quad n = 2,3,\ldots,
\end{align}
for a constant $C>0$. Then the chain $\{Y_n: n \geq 0\}$ fulfils \eqref{eq:DSP} with the same constant $C$.
\end{lemma}

\begin{proof}
Since $\bbP(\tau_2=1)=0$, for any $x,y \in X$ we have
\begin{align*}
\bbP (Y_2=y\mid Y_0=x) & = \sum_{k=1}^\infty \bbP (Z_k=y\mid Z_0=x)\bbP(\tau_2=k)\\
&\leq C \sum_{k=1}^\infty \bbP (Z_k=y\mid Z_0=x)\bbP(\tau_1=k) = C \bbP (Y_1=y\mid Y_0=x),
\end{align*}
as desired.
\end{proof}

With this result at hand one can effectively construct examples of Markov chains in $X$ that satisfy \eqref{eq:DSP} through choosing a random change of time process $\{\tau_n:n \geq 0 \}$ which satisfies \eqref{eq:DSP_N}. This observation provides an easy-to-check sufficient condition for assumption \eqref{A}.

\begin{corollary} \label{cor:sufficient}
If $\{Z_n: n \geq 0\}$ is irreducible, $\{\tau_n:n \geq 0 \}$ is such that \eqref{eq:DSP_N} holds and there exists $n_0 \in \N$ such that 
\begin{align}\label{eq:tau_pos}
\bbP(\tau_1=n)>0, \quad n \geq n_0,
\end{align}
then the subordinate chain $\{Y_n: n \geq 0\}$ satisfies assumption \eqref{A}.
\end{corollary}

\begin{proof}
In view of Lemma \ref{lem:subord}, we only need to show $\bbP (Y_1=y\mid Y_0=x)>0$, for all $x,y \in X$. 
Irreducibility of $\{Z_n: n \geq 0\}$  implies that for any $x,y\in X$ there exists $k \geq n_0$ such that 
$$
\bbP (Z_k=y\mid Z_0=x) >0.
$$
Then, by \eqref{eq:tau_pos},
\begin{align*}
\bbP (Y_1=y\mid Y_0=x) & = \sum_{n=1}^\infty \bbP (Z_n=y\mid Z_0=x)\bbP(\tau_1=n) \\
& \geq \bbP (Z_k=y\mid Z_0=x)\bbP(\tau_1=k) >0
\end{align*} 
and the result is proved.
\end{proof}

The class of processes $\{\tau_n:n \geq 0 \}$ that fits our assumptions is relatively large. Among other examples it includes random walks which are run by subexponential distributions, cf.\ \cite[Sec.\ 1.3.1]{Borowkow_book}, i.e.\ random walks satisfying
\begin{align}\label{eq:subexp}
\lim_{n\to \infty}\frac{\bbP (\tau_1 = n+1)}{\bbP (\tau_1 =n)}=1\quad \text{and}\quad 
\lim_{n\to \infty}\frac{\bbP (\tau_2=n)}{\bbP (\tau_1 =n)}=2.
\end{align} 

A useful sufficient condition for \eqref{eq:DSP_N} is given in the following lemma. 
We omit the proof as it is an easy modification of the argument from Corollary \ref{cor:JandK}.

\begin{lemma}\label{lem:suff_DSP_N}
Suppose that
$$
\bbP (\tau_1 =n) \asymp j(n)l(n), \quad n \in \N,
$$
where $j, l:\N \to (0,\infty)$ are non-increasing sequences such that for constants $C_1,C_2>0$
$$
j(n) \leq C_1 j(2n), \quad n \in \N,
$$
and
$$
l(n)l(m) \leq C_2 l(n+m), \quad n,m \in \N.
$$
Then $\{\tau_n:n \geq 0 \}$ satisfies  \eqref{eq:DSP_N}.
\end{lemma}

Another important examples of increasing random walks which can be used as a random change of time in the present framework are \textit{discrete subordinators} introduced in \cite{Bendikov-Saloff}.
Such processes correspond to Bochner's subordination which is a well-known concept in the theory of continuous time Markov processes. 
To be more precise, let $\phi$ be a Bernstein function \cite{Schilling_book} such that $\phi(0+)=0$, $\phi(1)=1$ and which admits the following representation
\begin{align*}
\phi (\lambda) = b\lambda +\int_{0}^{\infty} (1-e^{-\lambda t})\, \nu(\mathrm{d}t),
\end{align*} 
for a constant $b\geq 0$ and a measure $\nu$ on $(0,\infty)$ satisfying $\int(1\wedge x)\nu (\mathrm{d}x)<\infty$. Let $\{\tau_n:n \geq 0 \}$ be a random walk (discrete subordinator) taking values in $\mathbb{N}_0$, started at 0 and with the first-step-distribution given by
\begin{align}\label{discr_subord_formula}
\Pp (\tau_1 = k) = b\delta_1(k)+\frac{1}{k!}\int_0^\infty t^k e^{-\lambda t}\nu (\mathrm{d}t).
\end{align}

We remark that if $L$ is the discrete generator of the Markov chain $\{Z_n: n \geq 0\}$ then the generator of the subordinate process $\{Y_n: n \geq 0\}$ can be computed directly with the functional calculus and is equal to $-\phi(-L)$ (for details see \cite[Section 2.3]{Bendikov-Saloff}).  

The rest of this section is devoted  to a special situation where $\{Z_n: n \geq 0\}$ is assumed to be a nearest-neighbour (also called simple) random walk on a graph of finite geometry over $X$. By $L$ we denote the discrete Laplacian related to $\{Z_n: n \geq 0\}$ which is a local operator  in  the sense that $Lf(x)$ depends only on finitely many values of the function $f$ that are taken on vertices neighbouring to $x$. Recall that such processes can be seen as discrete-time counterparts of diffusions in $X$. As we mentioned in the Introduction, in this framework the concept of discrete subordination enables us to define numerous non-local discrete counterparts of operators which are known from the theory of jump L\'evy processes in Euclidean spaces. This includes fractional powers of the discrete Laplacian and quasi-relativistic operators which play an important role in various applications. We first discuss in more detail the properties of the corresponding discrete subordinators. 

\subsubsection*{Stable and relativistic stable subordinators} 
Let 
\begin{align*} 
\phi_m(\lambda)= \frac{(\lambda +m^{1/\alpha})^\alpha -m}{\theta_m},\quad  \text{for}\ \alpha \in (0,1)\ \text{and any}\ m\geq 0,
\end{align*}
where $\theta_m = (1 +m^{1/\alpha})^\alpha -m$.
We note that $\phi_m(\lambda)$ is a Bernstein function such that $\phi_m (1)=1$ and it admits the following L\'evy measure 
\begin{align*}
\nu_m (\mathrm{d}t)=
\frac{\alpha}{\theta_m \normal \Gamma (1-\alpha) }e^{-m^{1/\alpha}t}t^{-1-\alpha}\mathrm{d}t.
\end{align*}
Let $\{\tau_n^{(m)}:\, n\geq 0\}$ denote the corresponding discrete subordinator. For $m=0$ it is called the \textit{$\alpha$-stable} subordinator (observe that $\phi_0 (\lambda) =\lambda^\alpha$), while for $m>0$ it is called the \textit{relativistic $\alpha$-stable} subordinator.
%The terminate case $m=0$ leads to the \textit{$\alpha$-stable} subordinator with $\phi_0 (\lambda) =\lambda^\alpha$. Let $\{\tau_n^{(m)}:\, n\geq 0\}$ denote the corresponding discrete subordinator. 
With the aid of \eqref{discr_subord_formula} we easily find that
\begin{align}\label{tempered_prob}
a_m(k)\coloneqq \Pp (\tau_1^{(m)}=k) = \frac{\alpha}{\theta_m \normal \Gamma (1-\alpha) }\frac{\Gamma (k-\alpha)}{\Gamma (k+1)}(1+m^{1/\alpha})^{\alpha -k},\quad k\in \mathds{N},
\end{align}
which implies the following relation
\begin{align}\label{relation_tempered}
a_m(k)=\theta_m^{-1} e^{M(\alpha -k)}a_0(k),\quad M= \log (1+m^{^1/\alpha}) = \log(\theta_m+m)^{1/\alpha},
\end{align}
where 
\begin{align}\label{stable_prob}
a_0(k) = \frac{\alpha}{\Gamma (1-\alpha)}\frac{\Gamma (k-\alpha)}{\Gamma (k+1)}
\end{align}
is the first-step-distribution of the $\alpha$-stable discrete subordinator. 

It is also clear that \eqref{relation_tempered} extends to the convolution powers, that is
\begin{align*}
a_m^{\ast n}(k)=e^{(M\alpha -\log \theta_m)n}e^{-Mk}a_0^{\ast n}(k),\quad k, n \in \mathds{N}.
\end{align*} 
This equality reveals that the long range distributional properties of $\{\tau_n^{(m)}:\, n\geq 0\}$ for $m=0$ and $m>0$ are essentially different. 
\normal

We next show that the (relativistic) $\alpha$-stable subordinator enjoys the DSP property.

\begin{proposition} \label{prop:rel_stab}
 For any $m \geq 0$ and $\alpha \in (0,1)$ there is a constant $C=C(\alpha ,m)>0$ such that the (relativistic) $\alpha$-stable subordinator $\{\tau_n^{(m)}:\, n\geq 0\}$ satisfies \eqref{eq:DSP_N} with $C$.
\end{proposition}

\begin{proof}
We apply Wendel's bounds \cite{Wendel} in the form 
\begin{align*}
\Big( \frac{x}{x+s}\Big)^{1-s}\Gamma (x)\leq x^{-s}\Gamma (x+s)\leq \Gamma (x),\quad x>0,\ s\in (0,1).
\end{align*}
By setting $x = k-\alpha $ for any $k\in \mathds{N}$ and $s=\alpha$ we arrive at
\begin{align}\label{c_0_estimate}
\frac{1}{k^{\alpha +1}}\leq \frac{\Gamma (k-\alpha)}{\Gamma (k+1)}\leq \frac{1}{1-\alpha}\frac{1}{k^{\alpha +1}},\quad k\geq 1.
\end{align}
Together with \eqref{relation_tempered}--\eqref{stable_prob} this gives that 
$$
a_m(k)=e^{M(\alpha -k)}a_0(k) \asymp  j(k) l(k),
$$
with
$$
j(k) = k^{-1-\alpha}, \qquad l(k)= e^{-Mk}. 
$$
The assertion follows then from Lemma \ref{lem:suff_DSP_N}.
\end{proof}

\subsubsection*{Nearest-neighbour random walk and the corresponding subordinate Markov chain}
We now present useful estimates of the one-step transition probabilities for the subordinate nearest neighbour random walk where the underlying subordinator is (relativistic) $\alpha$-stable. We consider the graph $G$ over $X$ which satisfies our assumption \eqref{C} and further assume that $G$ is endowed with a family of symmetric and non-negative weights (conductances) $\left\{\mu_{x,y}\right\}_{x,y \in X}$ such that 
$\mu_{x,y} >0$ if and only if $ x \sim y$. 
We set $\mu_x = \sum_{y \in X} \mu_{x,y}$ and consider the measure on $X$ given by $\mu(A) = \sum_{x \in A} \mu_x$. 
The corresponding nearest-neighbour random walk $\{Z_n: n \geq 0\}$ is then a $\mu$-symmetric time-homogeneous Markov chain with values in $X$ and one-step transition probabilities given by
$$
\bbP (Z_{n+1}=y\mid Z_n=x):= \frac{\mu_{x,y}}{\mu_x},
$$
see e.g. Barlow \cite{Barlow_book} or Kumagai \cite{Kumagai}. 
Let
$$
g_n(x,y) = \frac{\bbP^x (Z_{n}=y)}{\mu_y}
$$
denote the transition densities of $\{Z_n: n \geq 0\}$ with respect to $\mu$. Finally, let $\{Y_n: n \geq 0\}$ be the subordinate Markov chain which is subordinated by the (relativistic) $\alpha$-stable subordinator $\{\tau_n^{(m)}:\, n\geq 0\}$. Recall that by \eqref{relation_tempered}, its one-step transition probabilities are given by $P(x,y) = p(x,y) \mu_y$, where 
$$ 
p(x,y)=\sum_{n=1}^\infty g_n(x,y) \Pp (\tau_1^{(m)}=n) = \frac{e^{M\alpha}}{\theta_m}\sum_{n=1 \vee d(x,y)}^\infty g_n(x,y)  e^{-Mn}a_0(n), \ M=\log (1+m^{1/\alpha}).
$$
We next find two-sided, qualitatively sharp estimates of the kernel $p(x,y)$. This can be done under the assumption that the densities $g_n(x,y)$ satisfy the sub-Gaussian estimates. More precisely, we assume that there are parameters $\beta, \gamma >1$ and the constants $c_1,\ldots,c_4>0$ such that for all $x, y \in X$, $n \in \N$,
\begin{align}\label{UHK}
g_n(x,y)
\leq 
\frac{c_1}{n^{\gamma /\beta}}
\exp\Big\{-c_2\left(\frac{d(x,y)}{n^{1/\beta}}\right)^{\frac{\beta}{\beta-1}}\Big\},
\end{align}
\begin{align}\label{LHK}
g_n(x,y)
+
g_{n+1}(x,y)
\geq
\frac{c_3}{n^{\gamma /\beta}}
\exp\Big\{-c_4\left(\frac{d(x,y)}{n^{1/\beta}}\right)^{\frac{\beta}{\beta-1}}\Big\},
\end{align}
whenever $n \geq d(x,y)$, cf.\ \cite[Definition 3.3.4(1)]{Kumagai} (note that for $\beta=2$ these are Gaussian bounds). For the lower bound we take into account the sum of $p_n(x,y)$ and $p_{n+1}(x,y)$ as it may happen that the graph $G$ is bipartite, as in the case of $\bbZ^d$. Upper and lower heat kernel bounds of the form \eqref{UHK}--\eqref{LHK} are valid on the  integer lattice but also on many fractal-type graphs including the famous example of the graphical Sierpinski gasket \cite{Jones} and Sierpinski carpet \cite{Barlow_Bass}, and more general graphs \cite{Barlow_book, Grigoryan-Telcs, Hambly_Kumagai, Kumagai}. 

\begin{proposition}\label{prop:UHK_relativistic}
Under \eqref{UHK}--\eqref{LHK} the following estimates holds.
\begin{itemize}
\item[a)] If $m=0$, then there is a constant $C \geq 1$ such that 
\begin{align*}
\frac{1}{C}\frac{1}{(1+d(x,y))^{\alpha \beta +\gamma}} \leq p(x,y)  \leq C \frac{1}{(1+d(x,y))^{\alpha \beta +\gamma}},\quad x,y\in X.
\end{align*}
\item[b)] If $m>0$, then there are constants $C, \widetilde C \geq 1$ such that
\begin{align*}
\frac{1}{C} \exp(-\widetilde C d(x,y)) \leq p(x,y)\leq C \exp\bigg( -\frac{d(x,y)}{\widetilde C}\bigg),\quad x,y\in X.
\end{align*}
\end{itemize}
\end{proposition}

\begin{proof}
We start with part a). It follows by \eqref{stable_prob}--\eqref{c_0_estimate} and \eqref{UHK}--\eqref{LHK} that
$$
p(x,x) \asymp \sum_{n=1}^\infty  n^{-\alpha -\gamma/\beta -1} < \infty, 
$$
and thus we only need to consider $x,y \in X$ for which $d(x,y) \geq 1$. For the upper bound, we observe that by \eqref{stable_prob}, \eqref{c_0_estimate} and \eqref{UHK} we have
\begin{align*}
p(x,y) & \leq c_1^\prime  \sum_{n=d(x,y)}^\infty  n^{-\alpha -\gamma/\beta -1}
\exp\Big(-c_2\left(\frac{d(x,y)}{n^{1/\beta}}\right)^{\frac{\beta}{\beta-1}}\Big)\\
& \leq c_3\int_{d(x,y)}^{\infty} t^{-\alpha -\gamma/\beta -1}
\exp\Big(-c_4\left(\frac{d(x,y)}{t^{1/\beta}}\right)^{\frac{\beta}{\beta-1}}\Big)dt.
\end{align*}
With the substitution $t=(u\, d(x,y))^{\beta}$, we obtain that the last integral is equal to
$$
\beta \, d(x,y)^{-\alpha \beta-\gamma} \int_{d(x,y)^{\frac{1}{\beta}-1}}^{\infty}
 u^{-\alpha \beta -\gamma -1}
\exp\Big(-c_4\left(\frac{
1}{u}\right)^{\frac{\beta}{\beta-1}}\Big)du,
$$
which leads to the desired bound 
$$
p(x,y) \leq c_5 d(x,y)^{-\alpha \beta-\gamma}.
$$
For the matching lower bound, we observe that by \eqref{stable_prob}, \eqref{c_0_estimate} and \eqref{LHK},
\begin{align*}
\sum_{n=d(x,y)}^\infty \big(g_n(x,y)+g_{n+1}(x,y)\big) a_0(n) 
& \geq c_6  \sum_{n=d(x,y)}^\infty  n^{-\alpha -\gamma/\beta -1}
\exp\Big(-c_7\left(\frac{d(x,y)}{n^{1/\beta}}\right)^{\frac{\beta}{\beta-1}}\Big)\\
& \geq c_6 e^{-c_7}  \sum_{n=\left\lceil  d(x,y)^{\beta}\right\rceil}^\infty  n^{-\alpha -\gamma/\beta -1} \\
& \geq c_8  (1+d(x,y))^{-\alpha\beta -\gamma}.
\end{align*}
On the other hand, by \eqref{tempered_prob} and \eqref{c_0_estimate},
$$
a_0(n) \asymp a_0(n+1), \quad n \in \N,
$$
which yields
$$
p(x,y) \geq c_9 (1+d(x,y))^{-\alpha\beta -\gamma},
$$
and the proof of part a) is completed.

To establish part b) we observe that the upper bound follows directly by \eqref{stable_prob}, \eqref{c_0_estimate} and \eqref{UHK},
$$
p(x,y) \leq c_{10}\frac{e^{M\alpha}}{\theta_m}\sum_{n=1 \vee d(x,y)}^\infty  e^{-Mn} a_0(n) n^{-\gamma/\beta}\leq c_{11} e^{-M d(x,y)}.
$$
The proof of the lower estimate is similar to that one in part a). Indeed, by \eqref{tempered_prob}, \eqref{stable_prob}, \eqref{c_0_estimate}, and \eqref{LHK}, we obtain
\begin{align*}
p(x,y)  &\geq c_{12}
\sum_{n=d(x,y)}^\infty e^{-Mn} n^{-\alpha -\gamma/\beta -1}
\exp\Big\{-c_7\left(\frac{d}{n^{1/\beta}}\right)^{\frac{\beta}{\beta-1}}\Big\} \\
&\geq 
c_{12} d(x,y)^{-\alpha -\gamma/\beta -1}  e^{-Md(x,y)} \exp\Big\{-c_7\left(\frac{d(x,y)}{d(x,y)^{1/\beta}}\right)^{\frac{\beta}{\beta-1}}\Big\} \\
&\geq 
c_{13} \exp\Big\{-c_{14} d(x,y)\Big\},
\end{align*}
and the proof is finished.
\end{proof}

\subsection{Markov chains with independent coordinates on product spaces}\label{sec:product} 
The following observation was kindly communicated to us by T.\ Kulczycki. Suppose we are given two independent Markov chains $\big\{Y^{(1)}_n: n \in \N_0\big\}$ and $\big\{Y^{(2)}_n: n \in \N_0\big\}$ with values in countably infinite spaces $X_1$ and $X_2$. Then the product chain $\big\{(Y^{(1)}_n,Y^{(2)}_n): n \in \N_0\big\}$ with values in $X_1 \times X_2$ satisfies the DSP if and only if each of its coordinates has this property. It easily extends to general product discrete time processes with finitely many independent coordinates and provides a lot of interesting examples of Markov chains with the DSP on integer lattices and products of more general graphs. 

Interestingly, this example allows us to observe that the class of Markov chains with the DSP includes also processes which neither have one-step transition probability with an isotropic profile on a countable measure metric space, nor are subordinate Markov chain, cf.\ Sections \ref{sec:chains_metric}, \ref{sec:chains_subordinate}.  It may lead to highly anisotropic transition probabilities, e.g.\ one can have 
$$P(x,y) = P_1(x_1,y_1)P_2(x_2,y_2),$$ 
where $P_1$ decays polynomially and $P_2$ decays exponentially, cf.\ Section \ref {sec:expl_rates} paragraphs (2) and (3).

\subsection{Estimates of harmonic functions -- a few explicit examples.}\label{sec:expl_rates}

We now give some examples of the decay rates for ($\calU -I$)-harmonic functions for various types of Markov chains with values in a countably infinite set $X$ for which the one-step transition probabilities $P(x,y)$ are driven by profiles with respect to a given metric $d$ on $X$. 
We analyze nearest-neighbor random walks and chains with strictly positive kernels $P(x,y)$ with polynomial and exponential decay at infinity. 

For better illustration we also assume that the confining potential $V$ takes the form $V(x) = W(d(x,x_0))$ for some $x_0 \in X$ with the profile function $W:[0,\infty) \to \R$ such that $\log W$ is an increasing function regularly varying (at infinity) of index $\rho \geq 0$. 

\smallskip 
\noindent
(1) \textbf{Nearest-neighbour random walk.} Our Corollary \ref{cor:nnrw_strongest} (resulting from Theorems \ref{th:upper_subharm_nnrw}-\ref{th:lower_est_nnrw}) states that in this case the decay rate of ($\calU -I$)-harmonic functions is governed by the expression 
\begin{align}\label{prof:1}
e^{-\frac{1}{1+\rho}d(x,x_0) \log W(d(x,x_0)) (1+o(1))}, \quad \text{as} \ \ d(x,x_0) \to \infty.
\end{align}
This is illustrated in Table \ref{Table_NNRW} for several typical profiles $W$.
Interestingly, we observe that in this case for confining potentials the decay rate is always super-exponential. 
 \begin{table}[h]
\centering
\begin{tabular}{|c| >{\centering} b{3.4cm}|  >{\centering} b{4.0cm}|b{4.3cm}<{\centering}|}
\hline
profile $W(n)$ & $\exp(c n^\rho)$ & $n^\rho$ & $(\log n)^\rho$  \\
\hline
decay rate \eqref{prof:1} & $e^{-\frac{c}{1+\rho} d(x,x_0)^{\rho +1}(1+o(1))}$\, & $e^{-\rho d(x,x_0) \log d(x,x_0)(1+o(1))}$\, & $e^{- d(x,x_0) \log\log d(x,x_0) (1+o(1))}$ \\
\hline
\end{tabular}
\caption{The case of nearest-neighbor walk ($\rho>0$)}
\label{Table_NNRW}
\end{table}

\vspace*{-0,7cm}
\noindent
(2) \textbf{Long-range random walks with polynomial transition probabilities.} Let us consider Markov chains with one-step transition probabilities satisfying 
$$
P(x,y)\asymp d(x,y)^{-\gamma},\quad x,y\in X, \quad x\neq y
$$
for some $\gamma >0$. This class includes some of examples discussed in Section \ref{sec:chains_metric} (see, e.g. \cite{Bass_Levin, Murugan_Saloff_2, Murugan_Saloff-Coste}) and the subordinate chains obtained for discrete $\alpha$-stable subordinators ($m=0$) presented in Section \ref{sec:chains_subordinate}, see Proposition \ref{prop:UHK_relativistic} a). The decay rate obtained for such chains in Theorem \ref{thm:main_1} and Proposition \ref{prop:lower_bound} (see also Corollary \ref{cor:BHI_infty}) takes the form 
\begin{align}\label{prof:2}
\frac{1}{d(x,x_0)^{\gamma} W(d(x,x_0))},  \quad \text{as} \ \ d(x,x_0) \to \infty.
\end{align}
The decay rates of ($\calU -I$)-harmonic functions corresponding to such Markov chains are illustrated in Table \ref{Table_Polyn} below.

 \begin{table}[h]
\centering
\begin{tabular}{|c| >{\centering} b{3.4cm}|  >{\centering} b{4.0cm}|b{4.3cm}<{\centering}|}
\hline 
profile $W(n)$ & $\exp(c n^\rho)$ & $n^\rho$ & $(\log n)^\rho$  \\
\hline
decay rate \eqref{prof:2} & $e^{-c d(x,x_0)^{\rho}}d(x,x_0)^{-\gamma}$\, & $d(x,x_0)^{-\gamma-\rho}$\, & $d(x,x_0)^{-\gamma}(\log n)^{-\rho} $ \\
\hline
\end{tabular}
\caption{The case of chains with polynomial transition probabilities ($\rho>0$).}
\label{Table_Polyn}
\end{table}

\noindent
(3) \textbf{Random walks with exponential transition probabilities.} Suppose there are $c_1, c_2>0$ such that
$$
c_1e^{-c_2 d(x,y)} \leq P(x,y) \leq c_3 e^{-c_4 d(x,y)} ,\quad x,y\in X,
$$
This covers chains with $P(x,y)$ as in Corollary \ref{cor:JandK} for $K(r)=e^{-cr}$ as well as subordinate chains obtained for discrete relativistic $\alpha$-stable subordinators ($m>0$) introduced in Section \ref{sec:chains_subordinate}, see Proposition \ref{prop:UHK_relativistic} b). As in (2), the decay rate obtained for this class (see Theorem \ref{thm:main_1}, Proposition \ref{prop:lower_bound} and Corollary \ref{cor:BHI_infty}) is
\begin{align}\label{prof:3}
e^{-\widetilde c d(x,x_0)} \frac{1}{ W(d(x,x_0))},  \quad \text{as} \ \ d(x,x_0) \to \infty,
\end{align}
where $\widetilde c=c_2$ in the lower bound and $\widetilde c=c_4$ in the upper bound. The behaviour of ($\calU -I$)-harmonic functions in this case is presented in Table \ref{Table_Exp}.

 \begin{table}[h]
\centering
\begin{tabular}{|c| >{\centering} b{3.4cm}|  >{\centering} b{4.0cm}|b{4.3cm}<{\centering}|}
\hline
profile $W(n)$ & $\exp(c n^\rho)$ & $n^\rho$ & $(\log n)^\rho$  \\
\hline
decay rate \eqref{prof:3} & $e^{-c d(x,x_0)^{\rho}-\widetilde c d(x,x_0)}$& $e^{-\widetilde c d(x,x_0)}d(x,x_0)^{-\rho}$\, & $e^{-\widetilde c d(x,x_0)}(\log n)^{-\rho} $ \\
\hline
\end{tabular}
\caption{The case of chains with exponential transition probabilities ($\rho>0$).}
\label{Table_Exp}
\end{table} 

\noindent \textbf{Acknowledgements}. We thank Krzysztof Bogdan, Tadeusz Kulczycki, Mateusz Kwa\'snicki and Ren\'e Schilling for discussions and helpful comments.

%%%%%%%%%%%%%%%%%%%%%%% BIBLIO
\bibliographystyle{abbrv}
\bibliography{discrete_f-k_v1}

\end{document}